\documentclass[12pt,reqno]{amsart}
\usepackage{amsmath,amssymb,amsfonts,amscd,latexsym,amsthm,mathrsfs,verbatim,textcomp}
\usepackage[colorlinks,linkcolor=black,citecolor=black]{hyperref}
\usepackage{cite}
\usepackage{hypbmsec}
\usepackage{enumerate}
\usepackage{bm}
\usepackage{tikz}
\usepackage{mathtools}
\usepackage{verbatim}
\usepackage{xfrac, mathabx}
\theoremstyle{plain}
\usepackage{manfnt}
\usepackage{tikz}
\usetikzlibrary{matrix,arrows,decorations.pathmorphing}
\usepackage[margin=1.1in]{geometry}
\usepackage{csquotes}
\newtheorem{theorem}{Theorem}
\newtheorem{lemma}[theorem]{Lemma}
\newtheorem{corollary}[theorem]{Corollary}
\newtheorem{prop}[theorem]{Proposition}
\theoremstyle{remark}
\newtheorem{remark}[theorem]{\bf Remark}

\newtheorem{conjecture}[theorem]{\bf Conjecture}

\newtheorem*{prob*}{\bf Problem}

\newtheorem*{quest*}{\bf Question}

\renewcommand{\Im}{\operatorname{Im}}
\renewcommand{\Re}{\operatorname{Re}}

\newcommand{\Z}{\mathbb{Z}}

\newcommand{\e}{\operatorname{e}}

\usepackage{etoolbox}
\patchcmd{\section}{\scshape}{\bfseries}{}{}
\makeatletter
\renewcommand{\@secnumfont}{\bfseries}
\makeatother

\numberwithin{theorem}{section}
\numberwithin{equation}{section}

\newcommand{\pMatrix}[4]{\left(\begin{matrix}#1 & #2 \\ #3 & #4\end{matrix}\right)}

\renewcommand{\pmatrix}[4]{\left(\begin{smallmatrix}#1 & #2 \\ #3 & #4\end{smallmatrix}\right)}

\makeatletter
\@namedef{subjclassname@2020}{\textup{2020} Mathematics Subject Classification}
\makeatother

\newcommand{\ppmod}[1]{\hspace{-0.15cm}\pmod{#1}}

\begin{document}
\author{Chantal David, Alexander Dunn, Alia Hamieh and Hua Lin}

\address{Chantal David: Department of Mathematics and Statistics, 
Concordia University, 1455 de Maisonneuve West, Montr\'{e}al, Qu\'{e}bec, Canada H3G 1M8}
\email{chantal.david@concordia.ca}

\address{Alexander Dunn: School of Mathematics, Georgia Institute of Technology, Atlanta, USA}
\email{ajd3303@gmail.com}

\address{Alia Hamieh: Department of Mathematics and Statistics, University of Northern British Columbia, Prince George,
BC V2N4Z9, Canada}
\email{alia.hamieh@unbc.ca}

\address{Hua Lin: Department of Mathematics, Northwestern University, 
2033 Sheridan Rd, Evanston, IL, 60208}
\email{hua.lin@northwestern.edu}


\title{Quartic Gauss sums over primes and metaplectic theta functions}
\maketitle
\begin{abstract}
We improve 1987 estimates of Patterson for sums of quartic Gauss sums over primes. 
Our Type-I and Type-II estimates feature new ideas, including use of the quadratic large sieve 
over $\mathbb{Q}(i)$, and Suzuki's evaluation of the Fourier-Whittaker coefficients of
quartic theta functions at squares. We also conjecture asymptotics for certain 
moments of quartic Gauss sums over primes.
\end{abstract}

\tableofcontents
\section{Introduction}

\subsection{Background and statement of results}
Gauss sums are fundamental objects in number theory.
Gauss \cite{Gauss} famously gave a closed form formula for the normalised quadratic Gauss sum over $\mathbb{Q}$
with prime modulus. Gauss' evaluation
depended only on the congruence class of the prime modulo $4$.

Gauss sums attached to characters of a fixed order greater than or equal to $3$ exhibit a very different behavior.
Kummer in 1846 \cite[Paper 16,17]{Kummer} studied the distribution of normalised cubic Gauss sums 
(over the Eisenstein quadratic field $\mathbb{Q}(\zeta_3)$) for small prime moduli, and predicted that the points landed
a ratio of $1:2:3$ when plotted on appropriate thirds of the unit circle.
In a breakthrough 1979 paper, Heath-Brown and Patterson \cite{HBP} 
disproved Kummer's conjecture and showed that the normalised cubic 
Gauss sums with prime moduli are equidistributed on the unit circle.
Patterson \cite{Pat3} generalised 
his equidistribution result with Heath-Brown to $n$th order Gauss sums to prime moduli 
(working over number fields $K$ with $K \supseteq \mathbb{Q}(\zeta_n)$ and $n \geq 3$ fixed).

Patterson \cite{Pat4}
(with refinements by Heath-Brown and Patterson \cite{HBP}) 
made a refined conjecture that concerned the asymptotic behavior for  
moments of normalised cubic Gauss sums. The conjectured main term in the first moment
explained the small number bias witnessed by Kummer in 1846.
Patterson's conjecture was established under the Generalised Riemann 
Hypothesis by the second author and Radziwi\l\l \cite{DR} in 2021, improving 
on an unconditional upper bound of Heath Brown \cite{HB} established in 2000.

The purpose of this paper is to improve the known estimates in the case of quartic Gauss sums over primes.
As explained in \S \ref{overall}, there are some significant differences when compared to the cubic case.

Before stating our result and conjecture, we need some notation. 
Let $e(x):=e^{2 \pi i x}$ for $x \in \mathbb{R}$ and
$\check{e}(z):=e(z+\overline{z})$ for $z \in \mathbb{C}$. 
Let $\mathbb{Q}(i)$ be the Gaussian quadratic field (class number $1$)
with ring of integers $\mathbb{Z}[i]$. Let $\lambda:=1+i$ denote the unique ramified prime.
For $c \in \mathbb{Z}[i]$ with $(c,\lambda)=1$, and $\nu \in \lambda^{-2} \mathbb{Z}[i]$,
the normalised quartic Gauss sum over $\mathbb{Z}[i]$
is defined as
\begin{equation*}
\widetilde{g}_4(\nu,c):=\frac{1}{\sqrt{N(c)}} \sum_{d \ppmod{c}} \Big( \frac{d}{c} \Big)_4 \check{e} \Big( \frac{\nu d}{c} \Big),
\end{equation*}
where $\big(\frac{\cdot}{c} \big)_4$ is the quartic symbol defined in \S \ref{quarticsec}.
We write $\widetilde{g}_4(c):=\widetilde{g}_4(1,c)$ for short. Let $\Lambda(c)$ denote the
von Mangoldt function on $\mathbb{Z}[i]$.

\begin{theorem} \label{mainthm} 
Let $R:(0,\infty) \rightarrow \mathbb{C}$ be a smooth function with compact support
in $[1,2]$. Then for any $\ell \in \mathbb{Z}$, $\beta \in \{1,1+\lambda^3 \} \ppmod{4}$, and $\varepsilon>0$, we have
\begin{equation*}
\sum_{\substack{c \in \mathbb{Z}[i] \\ c \equiv \beta \ppmod{4}}} \widetilde{g}_4(c) \Big( \frac{\overline{c}}{|c|} \Big)^{\ell}
\Lambda(c) R \Big( \frac{N(c)}{X} \Big) \ll_{\varepsilon,R} X^{5/6+\varepsilon} + X^{3/4+\varepsilon} |\ell| ^{3/2+\varepsilon},
\end{equation*}
as $X \rightarrow \infty$.
\end{theorem}
This substantially improves the bound $\ll_{\ell} X^{19/20}$ (for the sum with sharp cutoffs)
from \cite[Main Theorem]{Pat3}.

In analogy with Patterson's conjecture for the $X^{5/6}$-bias for the first moment of cubic Gauss sums \cite{Pat4,HBP}, 
we propose the following $X^{3/4}$-conjecture for the bias for quartic Gauss sums 
over primes.

\begin{conjecture} \label{quarticconj}
For $\beta \in \{1,1+\lambda^3\} \ppmod{4}$, there exists a constant
$b_{\beta} \neq 0$ such that for any $\varepsilon>0$ and
$\ell \in \mathbb{Z}$ we have
\begin{equation*}
\sum_{\substack{c \in \mathbb{Z}[i] \\ N(c) \leq X \\
c \equiv \beta \ppmod{4}}}
\widetilde{g}_4(c)
\Big( \frac{\overline{c}}{|c|} \Big)^{\ell}
\Lambda(c)
=
\begin{cases}
b_{\beta} X^{3/4}
+O_{\varepsilon}(X^{1/2+\varepsilon})
& \text{if} \quad \ell=0, \\
-\overline{b_{\beta}} X^{3/4}
+O_{\varepsilon}(X^{1/2+\varepsilon})
& \text{if} \quad \ell=1, \\
O_{\varepsilon,\ell}(X^{1/2+\varepsilon})
& \text{if} \quad \ell \notin \{0,1\},
\end{cases}
\end{equation*}
as $X \rightarrow \infty$.
\end{conjecture}

\begin{remark}
Using \eqref{momentcases} with $k=2$ and the Gauss sum identities
in \S~\ref{quarticsec}, we obtain, for primary $c$,
\[
\widetilde g_4(c)\frac{\overline c}{|c|}
=
\mu(c)\overline{\widetilde g_4(c)}.
\]
Since $\mu(\pi)=-1$ for primes $\pi$ and
$\widetilde g_4(\pi^k)=0$ for $k\geq2$, the $\Lambda$-weighted prime sum in
Conjecture \ref{quarticconj} for $\ell=1$ is exactly the negative complex conjugate
of the prime sum in that conjecture for $\ell=0$.
Thus the $\ell=1$ prediction follows from the $\ell=0$ prediction.
The Dirichlet series \eqref{dir} with $\nu=1$ runs over Gaussian integers
$c\equiv\beta\ppmod4$, including composites, not just primes.
The M\"obius factor therefore remains when $\ell=1$, so the conjectured prime bias
does not force a pole.
\end{remark}

The conjectured main term of size $\asymp X^{3/4}$ in the quartic case is substantially smaller than the main term of size $\asymp X^{5/6}$
in the cubic case \cite{DR}.
Conjecture \ref{quarticconj} is well out of reach even under the assumption of the 
Generalised Riemann Hypothesis for Hecke $L$-functions over $\mathbb{Q}(i)$ with Gr\"{o}ssencharakter.
One can see \S \ref{overall} for explanation of the bottleneck in the Type-II sums.

Before giving an overview of the paper in \S \ref{overall}, 
we briefly explain how the twisted sums that are the subject of Theorem \ref{mainthm} and Conjecture \ref{quarticconj}
are related to the equidistribution for \{$\widetilde{g}_4(\pi)\}_{\pi}$.
Weyl's criterion for equidistribution requires
that for each $k \in \mathbb{Z} \setminus \{0\}$,
\begin{equation} \label{moment}
\sum_{\substack{N(\pi) \leq X \\ \pi \equiv 1 \ppmod{\lambda^3}}} \widetilde{g}_4(\pi)^k=o_k \Big(\frac{X}{\log X} \Big) \quad \text{as} \quad
X \rightarrow \infty.
\end{equation}
We can discard the height two primes $\pi \equiv 1 \ppmod{\lambda^3}$
in \eqref{moment} i.e. $N(\pi)=p^2$ for $p \equiv 3 \ppmod{4}$ a rational prime,
since the contribution from such primes to \eqref{moment} is $O(X^{1/2})$. 
Without loss of generality we can assume the sum in \eqref{moment} is restricted to height one primes 
$\pi \equiv 1 \ppmod{\lambda^3}$ i.e $N(\pi)=p \equiv 1 \ppmod{4}$ is a rational prime.
The formulae \eqref{fourthpower} and \eqref{squarepower} imply that 
for height one primes 
$\pi \equiv 1 \ppmod{\lambda^3}$ and 
for $k \in \mathbb{Z} \setminus \{0\}$,
\begin{equation} \label{momentcases}
\widetilde{g}_4(\pi)^{k}=
\Big( \frac{\overline{\pi}}{|\pi|} \Big)^{-2 m}
\cdot
\begin{cases}
1 & \text{if} \quad k \equiv 0 \ppmod{4} \quad \text{with} \quad m=\tfrac{k}{4} \\
\widetilde{g}_4(\pi) & \text{if} \quad k \equiv 1 \ppmod{4} \quad \text{with} \quad m=\tfrac{k-1}{4} \\
-\big( \frac{\overline{\pi}}{|\pi|}  \big)^{-1}  \big(\frac{-1}{\pi} \big)_4\big( \frac{\overline{\pi}}{\pi}  \big)_4^{-2}   & \text{if} \quad k \equiv 2 \ppmod{4} \quad \text{with} \quad m=\tfrac{k-2}{4}  \\
-  \big( \frac{\overline{\pi}}{|\pi|}  \big)^{-1} \big(\frac{-1}{\pi} \big)_4\big( \frac{\overline{\pi}}{\pi}  \big)_4^{-2}\widetilde{g}_4(\pi) & \text{if} \quad k \equiv 3 \ppmod{4} \quad \text{with} \quad m=\tfrac{k-3}{4}
\end{cases}.
\end{equation}
From \eqref{momentcases} it follows that if $k \equiv 0,2 \ppmod{4}$, then \eqref{moment} is a standard consequence of
the zero-free region for $L$-functions with Gr\"{o}ssencharaktern. If 
$k \equiv 1 \ppmod{4}$, then it suffices to prove that for $\ell \in \mathbb{Z}$ and $\beta \in \{1,1+\lambda^3 \} \ppmod{4}$,
\begin{equation}\label{sum-of-gauss-sums-primes}
\sum_{\substack{c \in \mathbb{Z}[i] \\ c \equiv \beta \ppmod{4} \\ N(c) \leq X}} \widetilde{g}_4(c) \Big( \frac{\overline{c}}{|c|} \Big)^{\ell}
\Lambda(c)=o_{\ell} \Big ( \frac{X}{\log X} \Big).
\end{equation}
If $k \equiv 3 \ppmod{4}$, then we consider  \eqref{sum-of-gauss-sums-primes} with each term in the sum multiplied by $\big(\frac{\overline{c}}{c} \big)_4^{-2}$. We remark that the Gr\"{o}ssencharakter $\big(\frac{-1}{c} \big)_4$ is constant along the arithmetic progression $c \equiv \beta \ppmod{4}$. For simplicity we focus only on the case $k \equiv 1 \ppmod{4}$. 

\subsection{Overall strategy of the paper} \label{overall}
In this section we explain the main ideas of the paper without paying attention to technicalities.
Here we assume that quartic reciprocity over $\mathbb{Q}(i)$ is perfect (i.e. ignore congruence conditions modulo $4$), 
pretend that $\Gamma_1(\lambda^4) \lhd \operatorname{SL}_2(\mathbb{Z}[i])$ has one cusp, and so forth.
This section can be read independently of the rest of the paper.

For simplicity we will take $\ell=0$ in this sketch.
After detecting primes using Vaughan's identity \cite{Vau} we are faced with the sums,
\begin{align}
\sum_{\substack{ N(a) \sim A \\ N(b) \sim B \\ a,b \equiv 1 \ppmod{\lambda^3} }}
\rho_a \widetilde{g}_4(ab) 
R \Big( \frac{N(ab)}{X} \Big);  \label{t1} \\
\sum_{\substack{ N(a) \sim A \\ N(b) \sim B \\ a,b \equiv 1 \ppmod{\lambda^3}}} 
\rho_a \phi_b  \widetilde{g}_4(ab) R \Big( \frac{N(ab)}{X} \Big), \label{t2}
\end{align}
where $A,B \geq 1$, and
$\boldsymbol{\rho}=(\rho_{a})$ and $\boldsymbol{\phi}=(\phi_{b})$
are arbitrary $\mathbb{C}$-valued sequences. 
Both sums \eqref{t1} and \eqref{t2} 
are supported on $a,b \in \mathbb{Z}[i]$ such that $\mu^2(ab)=1$.
The sum in \eqref{t1} (resp. \eqref{t2}) is known as a Type-I
(resp. Type-II) sum. The decomposition of sums involving the von Mangoldt function into these 
types of multilinear sums can be found in \S \ref{vauidsec}.

The Chinese remainder theorem, quartic reciprocity, and change of variables
(see \eqref{rel1}--\eqref{rel3}) imply that (ignoring the factor involving $C(a,b)$ from \eqref{rel3}),
\begin{align} \label{tm}
\widetilde{g}_4(ab) &=
\widetilde{g}_4 (a)  \widetilde{g}_4 (b)  
\Big(\frac{a}{b} \Big)_2=\widetilde{g}_4(a) \widetilde{g}_4(a^2,b).
\end{align}
After applying the first equality in \eqref{tm} to \eqref{t2}, and relabelling the weights 
$\boldsymbol{\rho}=(\rho_{a})$ and $\boldsymbol{\phi}=(\phi_{b})$ appropriately,
the Type-II sum becomes
\begin{equation} \label{t22}
\sum_{\substack{ N(a) \sim A \\ N(b) \sim B \\ a,b \equiv 1 \ppmod{\lambda^3}}} 
\rho_a \mu^2(a) \phi_b \mu^2(b) \Big(\frac{a}{b} \Big)_2 R \Big( \frac{N(ab)}{X} \Big).
\end{equation}
After Mellin inversion of $R$,
the quadratic large sieve over $\mathbb{Q}(i)$ due to Onodera \cite{On} and Goldmakher--Louvel \cite{GL} can be readily applied to \eqref{t22}. This
yields an upper bound $\ll X^{5/6+\varepsilon}$ for $A \in [X^{1/3},X^{2/3}]$.
 The quadratic large sieve is recorded in \S \ref{quadsievesec}, and our
bounds for the Type-II sum appear in \S \ref{type2sec}.

The Type-I sums are more intricate than their counterparts in the cubic case \cite{DR,HB}.
After applying the second equality in \eqref{tm} to \eqref{t1},
and relabelling the weight $\boldsymbol{\rho}=(\rho_{a})$ appropriately,
the Type-I sum becomes
\begin{equation}  \label{t12}
\sum_{\substack{ N(a) \sim A \\ a \equiv 1 \ppmod{\lambda^3} }} 
\rho_a \mu^2(a) \sum_{\substack{N(b) \sim B \\ (b,a)=1  \\ b \equiv 1 \ppmod{\lambda^3} }}
\widetilde{g}_4(a^2,b)
R \Big( \frac{N(ab)}{X} \Big).
\end{equation}
After Mellin inversion of $R$ and shifting the contour to $\Re(s)=1/2+\varepsilon$,
one needs to understand the Dirichlet series,
\begin{equation*}
\widetilde{\psi}^{(4)}(s,\nu,0):=\sum_{\substack{c \in \mathbb{Z}[i] \\ c \equiv 1 \ppmod{\lambda^3}}} \frac{\widetilde{g}_4(\nu,c)}{N(c)^{s}}, \quad \text{for} \quad \Re(s)>1 \quad \text{and} \quad \nu \equiv 1 \ppmod{\lambda^3}.
\end{equation*}
These series have meromorphic continuation to all $\mathbb{C}$ and satisfy a functional equation $s \rightarrow 1-s$
because they essentially occur as the Fourier-Whittaker coefficients of Kubota's Eisenstein series $E^{(4)}(w,s,0)$, $w \in \mathbb{H}^3$, on the $4$-fold cover of $\operatorname{SL}_2(\mathbb{Z}[i])$
\cite{Kub1,Kub2,Dia}.
The function $\widetilde{\psi}^{(4)}(s,\nu,0)$ has at most 
a simple pole at $s=3/4$ in the half-plane $\sigma>1/2$, with residue denoted $\psi^{(4)}(\nu)$. For $\varepsilon>0$, the convexity bound is
\begin{align} \label{convexsketch}
\widetilde{\psi}^{(4)}&(s,\nu,0)  \ll N(\nu)^{(1/2)(1-\sigma)+\varepsilon} (|s|+1)^{3(1-\sigma)+\varepsilon}, \nonumber \\
& \text{for} \quad \sigma:=\Re(s) \geq 1/2+\varepsilon, \quad |s-3/4|>1/8,  
\quad \text{and} \quad \nu \equiv 1 \ppmod{\lambda^3}.
\end{align}
More details concerning Eisenstein series and these Dirichlet series can be found in \S \ref{dirsec}--\S \ref{metaeissec}.

Understanding the pole term in Type-I is subtle, 
and the occurrence of the $\widetilde{g}_4(a^2,b)$ in \eqref{t12} is crucial.
The residues $\psi^{(4)}(\nu)$ appear as the Fourier-Whittaker coefficients of the quartic theta function 
$\vartheta^{(4)}(w)$ (see \eqref{thetafourier} and \eqref{taurel}), and are in general not all determined!
This is in stark contrast to the case of the cubic theta function $\vartheta^{(3)}(w)$, where all the Fourier-Whittaker coefficients were
computed in closed form by Patterson \cite{Pat1}. Hecke theory and periodicity for the Fourier coefficients of theta functions
on the $n$-fold cover of $\operatorname{GL}_2$ ($n \geq 2$) essentially reduces the problem to determining
the Fourier coefficients at indices $1,\pi,\ldots,\pi^{n-1}$, for each $\pi$ prime. In general there are 
$n/2 -1$ (respectively $(n-1)/2 -1$) undetermined coefficients among this collection when $n$ is even (respectively odd),
one can consult the heuristic \cite[pg.75--78]{Hoff} and \cite[pg.~1904--1907]{BrHo}.
Returning to the quartic case $n=4$, Suzuki \cite{Suz1} managed to determine partial information, such as 
\begin{equation}
\psi^{(4)}(a^2)= \frac{\overline{ \widetilde{g}_4(a)}}{N(a)^{1/4}}, \quad \text{for} \quad a \equiv 1 \ppmod{\lambda^3} \quad \text{squarefree}.
\end{equation}
Thus $\psi^{(4)}(a^2) \ll N(a)^{-1/4}$, and this
leads to a contribution of $X^{3/4+\varepsilon}$ to Type-I. 
Note that $\nu \rightarrow a^2$ in \eqref{convexsketch} gives 
an increase in conductor for the relevant Dirichlet series in the $a$-aspect that does not occur in the cubic case \cite{HB},
and it is insufficient for Theorem \ref{mainthm} to use only
the convex bound $\psi^{(4)}(a^2) \ll N(a)^{1/4+\varepsilon}$ as in \cite{Pat3}. More information about metaplectic theta functions, their history, and 
Patterson's conjectures about them
appear in \S \ref{thetasec} and the references therein.

To handle the contour integral near the critical line,
we use an approximate functional equation and the quadratic large sieve in \S \ref{avgtype1sec}
to essentially prove that 
\begin{equation} \label{secondmomentsketch}
\sum_{\substack{N(a) \leq A \\ a \equiv 1 \ppmod{\lambda^3}}} \mu^2(a) |\widetilde{\psi}^{(4)}(1/2+\varepsilon+it,a^2,0) |^2 \ll_{\varepsilon} A^{3/2+\varepsilon} (|t|+1)^{3+\varepsilon}.
\end{equation}
The bound in \eqref{secondmomentsketch} saves $A^{1/2}$ over the convexity bound on average,
and leads to a contribution 
of $X^{1/2+\varepsilon} A^{3/4} \ll X^{3/4+\varepsilon}$ to Type-I
for level of distribution $A \leq X^{1/3}$. This is sufficient for our Theorem \ref{mainthm}.
With additional work (see Remark \ref{lindelofremark}), it seems likely that one could prove a 
Lindel\"{o}f-on-average (in $a$) bound $\ll A^{1+\varepsilon} (|t|+1)^{3+\varepsilon}$ in \eqref{secondmomentsketch},
and this would lead to a contribution of $X^{1/2+\varepsilon} A^{1/2} \ll X^{3/4+\varepsilon}$ to Type-I
for an increased level of distribution $A \leq X^{1/2}$.
This makes no significant difference in view of the $X^{3/4+\varepsilon}$ contribution from the pole term,
and the $X^{5/6+\varepsilon}$-bottleneck in Type-II, and so
we refrain from this additional work. 

With more effort one could remove the smoothing in Theorem \ref{mainthm}
at a power saving loss using the
hybrid large sieve of Jutila \cite{Jut} (adapted to $\mathbb{Q}(i)$). We opt for the best possible exponent within current technology using a smoothing.

It seems quite difficult to obtain an asymptotic upper bound of $X^{5/6+\varepsilon}$-quality for sums of $n$th order Gauss sums over primes
with $n \geq 5$ fixed. The main obstruction is that the relevant Dirichlet series in Type-I has Dirichlet coefficients $\widetilde{g}_n(a^{n-2},c)$,
and consequently has a functional equation that behaves like a $\operatorname{GL}_{n-2}$-object in $a$ (leading to an increased conductor and difficulty in controlling
the error terms in Type-I).

\subsection{Acknowledgements} 
This paper grew out of discussions between the four authors at the ANTeater’22 summer school. We thank 
the organizer Alexandra Florea and UC Irvine for their hospitality. We thank the referee for meticulous
comments on our manuscript, and also thank Maksym Radziwiłł for helpful comments. Dunn thanks
Adrian Diaconu and Jeffrey Hoffstein for helpful conversations regarding metaplectic theta functions.
We thank Joseph S. Farland for pointing out the relation between
the $\ell=0$ and $\ell=1$ prime sums and the resulting correction
to Conjecture \ref{quarticconj}.

\section{Conventions}
Dependence of implied constants on parameters will be indicated in the statements of results,
but partially suppressed throughout the body of the paper i.e. in the proofs. Implied constants in the body of the paper are
allowed to depend on $\varepsilon>0$ (possibly different in each instance), 
and the smooth compactly supported function $R:(0,\infty) \rightarrow \mathbb{C}$.
Other important dependencies will be indicated.

For $a \in \mathbb{N}$ and $A>0$, we use $a \sim A$ to mean $A<a \leq 2A$. 
Whenever we write $r \mid q$ with $0 \neq r,q \in \mathbb{Z}[i]$ and $q \equiv 1 \ppmod{\lambda^{3}}$,
it is our convention that $r \equiv 1 \ppmod{\lambda^3}$.

\section{Quartic Gauss sums} \label{quarticsec}
Our basic references for this section are \cite[Chap.~9]{IR} and \cite[\S 2]{Dia}.
Let $\mathbb{Q}(i)$ be the Gaussian imaginary quadratic field with ring of integers $\mathbb{Z}[i]$,
and 
$N(x):=N_{\mathbb{Q}(i)/\mathbb{Q}}(x)=|x|^2$ denote the norm function.
It is well known that $\mathbb{Q}(i)$ has class number one, discriminant $-4$, and 
unit group $\mathbb{Z}[i]^{\times}=\{\pm 1, \pm i \}$. 
Let $\lambda:=1+i$, the unique prime lying above $2$ in $\mathbb{Q}(i)$.
Each ideal $0 \neq \mathfrak{c} \unlhd \mathbb{Z}[i]$
is principal, and if $(\mathfrak{c},\lambda)=1$, then $\mathfrak{c}$ has a unique primary generator 
i.e. $\mathfrak{c}=(c)$ with $c \equiv 1 \ppmod{\lambda^3}$.
Each $c \equiv 1 \ppmod{\lambda^3}$ satisfies either 
$c \equiv 1 \ppmod{4}$ or $c \equiv 1+\lambda^3 \ppmod{4}$. The rational prime $2=-i \lambda^2$ is the only 
rational prime that ramifies in $\mathbb{Z}[i]$.
If $p \equiv 1 \ppmod{4}$ is a rational prime, then $p=\pi \overline{\pi}$ in $\mathbb{Z}[i]$
where $\pi, \overline{\pi} \in \mathbb{Z}[i]$ are distinct primes and
$N(\pi)=N(\overline{\pi})=p$. If
$p \equiv 3 \ppmod{4}$ is a rational prime, then $p=\pi$ remains inert in $\mathbb{Z}[i]$
and $N(\pi)=p^2$. Thus we have $N(\pi) \equiv 1 \ppmod 4$ for all primes $\pi$ with $(\pi) \neq (\lambda)$.

For any prime $\pi \in \mathbb{Z}[i]$ with $(\pi) \neq (\lambda)$ and $\alpha \in \mathbb{Z}[i]$ with
$(\alpha,\pi)=1$, the quartic (resp. quadratic) residue symbol is defined by
\begin{equation*}
\Big(\frac{\alpha}{\pi}  \Big)_4  \equiv \alpha^{\frac{N(\pi)-1}{4}} \ppmod{\pi}, 
\quad \text{resp.} \quad \Big(\frac{\alpha}{\pi}  \Big)_2  \equiv \alpha^{\frac{N(\pi)-1}{2}} \ppmod{\pi}.
\end{equation*}
The quartic symbol (resp. quadratic symbol) takes values in $\{\pm 1, \pm i \}$ (resp. $\{\pm 1\}$).
For $\alpha,\gamma \in \mathbb{Z}[i]$ with $(\gamma,\lambda)=1$, $(\alpha,\gamma)=1$, and 
prime factorisation $\gamma=\pi_1 \cdots \pi_k$, one extends these symbols multiplicatively,
\begin{equation*}
\Big( \frac{\alpha}{\gamma} \Big)_4:= \prod_{i=1}^{k} \Big(\frac{\alpha}{\pi_i} \Big)_4, \quad \text{resp.} \quad 
\Big( \frac{\alpha}{\gamma} \Big)_2:= \prod_{i=1}^{k} \Big(\frac{\alpha}{\pi_i} \Big)_2.
\end{equation*}
Furthermore,
\begin{equation*}
\Big(\frac{\alpha}{\gamma} \Big)_4=\Big(\frac{\alpha}{\gamma} \Big)_2:=0 \quad \text{for} \quad
\alpha, \gamma \in \mathbb{Z}[i] \quad \text{s.t.} \quad (\alpha,\gamma) \neq 1,
\quad (\gamma,\lambda)=1. 
\end{equation*}
Observe that 
\begin{equation} \label{quartquad}
\Big( \frac{\alpha}{\gamma} \Big)_4^2=\Big( \frac{\alpha}{\gamma} \Big)_2 \quad \text{for} \quad \alpha,\gamma \in \mathbb{Z}[i] 
\quad \text{s.t.} \quad (\gamma,\lambda)=1.
\end{equation}

For 
$\alpha,\gamma \in \mathbb{Z}[i]$ with $\alpha \equiv \gamma \equiv 1 \ppmod{\lambda^3}$
and $(\alpha,\gamma)=1$, 
the biquadratic reciprocity law \cite[\S 9.9, Theorem~2]{IR}  states that
\begin{equation} \label{bilaw}
\Big( \frac{\alpha}{\gamma} \Big)_4=(-1)^{C(\alpha,\gamma)} \Big( \frac{\gamma}{\alpha} \Big)_4,
\end{equation}
where 
\begin{equation} \label{Cdef}
C(\alpha,\gamma)=\frac{(N(\alpha)-1)}{4} \frac{(N(\gamma)-1)}{4}.
\end{equation}
There are supplementary laws for the ramified primes and units.
If
\begin{equation*}
\gamma=1+a_4 \lambda^4+a_5 \lambda^5+a_6 \lambda^6+\cdots \quad \text{with} \quad a_i \in \{0,1\},
\end{equation*}
we have
\begin{equation*}
\Big(\frac{\lambda}{\gamma} \Big)_4=i^{-a_4+2a_6} \quad \text{and} \quad \Big( \frac{i}{\gamma} \Big)_4=(-1)^{a_4+a_5}.
\end{equation*}
If 
\begin{equation*}
\gamma=1+\lambda^3+a_4 \lambda^4+a_5 \lambda^5+a_6 \lambda^6+\cdots \quad \text{with} \quad a_i \in \{0,1\},
\end{equation*}
then  
\begin{equation*}
\Big( \frac{\lambda}{\gamma} \Big)_4=-i^{-a_4+2a_6} \quad \text{and} \quad \Big( \frac{i}{\gamma} \Big)_4=i (-1)^{a_4+a_5}.
\end{equation*}
Note that the formula for $\big( \frac{i}{\gamma} \big)_4$ in the display immediately above corrects a typographic error in
\cite{Suz1} and \cite{Dia}.

Let $e(x):=e^{2 \pi i x}$ for $x \in \mathbb{R}$ and
$\check{e}(z):=e(z+\overline{z})$ for $z \in \mathbb{C}$. 
For $c \in \mathbb{Z}[i]$ with $(c,\lambda)=1$, and $\nu \in \lambda^{-2} \mathbb{Z}[i]$,
the quartic Gauss sum over $\mathbb{Z}[i]$
is defined as
\begin{equation*}
g_4(\nu,c):=\sum_{d \ppmod{c}} \Big( \frac{d}{c} \Big)_4 \check{e} \Big( \frac{\nu d}{c} \Big).
\end{equation*}
We write
$g_4(c):=g_4(1,c)$.
For $(r,c)=1$,
\begin{equation} \label{rel1}
g_4(r \nu,c)=\overline{\Big( \frac{r}{c} \Big)_4} g_4(\nu,c).
\end{equation}
For $\nu,c,c^{\prime} \in \mathbb{Z}[i]$ with $(cc^{\prime},\lambda)=1$ and $(c,c^{\prime})=1$, 
the Chinese Remainder Theorem implies the twisted multiplicativity relation
\begin{equation} \label{rel2}
g_4(\nu,cc^{\prime})= \Big( \frac{c}{c^{\prime}} \Big)_4 \Big( \frac{c^{\prime}}{c} \Big)_4 g_4(\nu,c) g_4(\nu,c^{\prime}).
\end{equation}
Note that for $(\nu,c c^{\prime})=1$ the relation \eqref{rel2} also holds for
$(c,c^{\prime}) \neq 1$ (both sides are $0$).
It follows from \eqref{bilaw}, \eqref{rel1}, and \eqref{rel2} that for $(c,c^{\prime})=1$ and 
$c \equiv c^{\prime} \equiv 1 \pmod{\lambda^3}$ we have
\begin{equation} \label{rel3}
g_4(\nu,cc^{\prime})=(-1)^{C(c^{\prime},c)} g_4(\nu (c^{\prime})^2,c) g_4(\nu,c^{\prime}).
\end{equation}
By \eqref{rel1} and \eqref{rel2} it suffices to understand $g_4(\pi^k,\pi^{\ell})$ for primes
$\pi \equiv 1 \ppmod{\lambda^3}$ and $k,\ell \in \mathbb{N}_{\geq 0}$.
Local computations yield 
\begin{equation} \label{rel4}
g_4(\pi^k,\pi^{\ell})=\begin{cases}
1 & \text{if } \ell=0, \\
N(\pi)^k g_4(\pi) & \text{if } \ell=k+1, \quad k \equiv 0 \ppmod{4},\\
N(\pi)^k g_2(\pi) & \text{if } \ell=k+1, \quad k \equiv 1 \ppmod{4}, \\ 
N(\pi)^k \big(\frac{-1}{\pi} \big)_4 \overline{g_4(\pi)} & \text{if } \ell=k+1, \quad k \equiv 2 \ppmod{4}, \\
-N(\pi)^k & \text{if } \ell=k+1, \quad k \equiv 3 \ppmod{4}, \\
\varphi(\pi^{\ell}) & \text{if } k \geq \ell, \quad \ell \equiv 0 \ppmod{4}, \\
0 & \text{otherwise},
\end{cases}
\end{equation}
where
\begin{equation*}
g_2(\nu,c):=\sum_{d \ppmod{c}} \Big( \frac{d}{c} \Big)_2 \check{e} \Big( \frac{\nu d}{c} \Big),
\end{equation*}
is the (auxiliary) quadratic Gauss sum.
We write 
$g_2(c):=g_2(1,c)$. 

For degree $2$ primes in $\mathbb{Z}[i]$ i.e. $\pi \equiv 1 \ppmod{\lambda^3}$ prime such that 
$N(\pi)=p^2$ with $p \equiv 3 \ppmod{4}$ a rational prime, 
we have that
\begin{equation} \label{trivquadquar}
\left|g_4(\pi)\right|=\left|g_2(\pi)\right|=p=N(\pi)^{1/2}.
\end{equation}
If $\pi \equiv 1 \ppmod{\lambda^3}$ is a degree $1$ prime in $\mathbb{Z}[i]$, then 
Gauss' evaluation is
\begin{equation} \label{quad2}
 g_2(\pi)=\Big(\frac{-1}{\pi} \Big)_4 N(\pi)^{1/2}.
\end{equation}
If $\pi \equiv 1 \ppmod{\lambda^3}$ is a degree $1$ prime in $\mathbb{Z}[i]$,
then we have the fourth power formula \cite[Proposition~9.9.5]{IR},
\begin{equation} \label{fourthpower}
g_4(\pi)^4= \pi^3 \overline{\pi}.
\end{equation}
and the square formula \cite[Proposition~9.10.1]{IR} implies that,
\begin{equation} \label{squarepower}
g_4(\pi)^2=-\Big( \frac{-1}{\pi} \Big)_4\Big( \frac{\overline{\pi}}{\pi}  \Big)_4^{-2}  N(\pi)^{1/2} \pi.
\end{equation}
Observe that \eqref{rel1}--\eqref{rel2}, \eqref{rel4}--\eqref{trivquadquar}, and \eqref{fourthpower}
imply that for $c \equiv 1 \ppmod{\lambda^3}$ we have
\begin{equation} \label{sqrootcancel}
|g_4(c)|=\mu^2(c) N(c)^{1/2}.
\end{equation}
We denote the normalised quartic Gauss sum (with shift) by
\begin{equation} \label{tildedef}
\widetilde{g}_4(\nu,c):=N(c)^{-1/2} g_4(\nu,c).
\end{equation}

\section{Metaplectic Eisenstein series and theta functions} \label{metasec}
Here we introduce the Dirichlet series formed from quartic Gauss sums and explain
their connection to metaplectic Eisenstein series on the 4-fold cover of $\operatorname{GL}_2$ over $\mathbb{Q}(i)$
and their theta functions.  These notions were first introduced by Kubota \cite{Kub1,Kub2}.

\subsection{Dirichlet series} \label{dirsec}
For $0 \neq \nu \in \lambda^{-2} \mathbb{Z}[i]$, $\ell \in \mathbb{Z}$, and $\beta \in \{1,1+\lambda^3 \} \ppmod{4}$,
consider the Dirichlet series,
\begin{equation} \label{dir}
\psi^{(4)}_{\beta}(s,\nu,\ell):=\sum_{\substack{c \in \mathbb{Z}[i] \\ c \equiv \beta \ppmod{4}}} 
\frac{g_4(\nu,c) ( \frac{\overline{c}}{|c|})^{\ell} }{N(c)^s}. 
\end{equation}
These series converge absolutely for $\sigma:=\Re(s)>3/2$. We write
$\psi^{(4)}_{\beta}(s,\nu,0):=\psi^{(4)}_{\beta}(s,\nu)$.
The functions $\psi^{(4)}_{\beta}(s,\nu,\ell)$
have meromorphic continuation to all of $\mathbb{C}$.
Both families are holomorphic in the half-plane $\sigma>1$, unless 
$\ell=0$, when $\psi^{(4)}_{\beta}(s,\nu)$ has at most a possible simple pole at $s=5/4$
in the half-plane $\sigma>1$. For $0 \neq \nu \in \lambda^{-2} \mathbb{Z}[i]$, let
\begin{equation} 
\psi^{(4)}_{\beta}(\nu):=\mathrm{Res}_{s=5/4} \psi^{(4)}_{\beta}(s,\nu). \label{res1}  
\end{equation}
These residues are shrouded in mystery.
More details on the meromorphic continuation and residues are provided in 
\S \ref{metaeissec} and \S \ref{thetasec} respectively.

\begin{remark}
Suzuki \cite[(1.1)]{Suz1} normalised the Dirichlet series $\psi^{(4)}_{1+\lambda^3}(s,\nu,0)$ with an extra $-$ sign.
In terms of sign normalisation we follow Diaconu \cite[(2.14)]{Dia} and do not include this normalisation in our 
definition \eqref{dir} for $\beta=1+\lambda^3$. 
We account for this sign discrepancy when recording various results of Suzuki \cite{Suz1} later in this Section.
\end{remark}

\subsection{Group action on $\mathbb{H}^3$ and Laplacian} \label{H3sec}
Let $\mathbb{H}^{3}$ denote the hyperbolic 3-space $\mathbb{C} \times \mathbb{R}^{+}$.
Embed $\mathbb{C}$ and $\mathbb{H}^3$ in the Hamilton quaternions  by identifying 
$i=\sqrt{-1}$ with $\hat{i}$ and 
$w=(z(w),v(w))=(z,v)=(x+iy,v) \in \mathbb{H}^3$ with $x+y \hat{i}+v \hat{j}$, where 
$1,\hat{i},\hat{j},\hat{k}$ denote the unit quaternions.  
A point $w=(z,v)$ is represented by the matrix $w=\pmatrix z {-v} {v} {\overline{z}}$,
and $u \in \mathbb{C}$ is represented by the matrix $\widetilde{u}=\pmatrix u 0 0 {\overline{u}}$.
The continuous action 
of $\operatorname{SL}_2(\mathbb{C})$ on $\mathbb{H}^3$
is given by
\begin{equation*}
\gamma w=(\widetilde{a}w+\widetilde{b})(\widetilde{c}w+\widetilde{d})^{-1}, \quad \gamma=\begin{pMatrix}
a b
c d 
\end{pMatrix} \in \operatorname{SL}_2(\mathbb{C}), \quad \text{and} \quad
w \in \mathbb{H}^3.
\end{equation*}
In coordinates,
\begin{equation*}
\gamma w= \bigg( \frac{(az+b) \overline{(cz+d)}+a \overline{c} v^2}{|cz+d|^2 +|c|^2 v^2}, 
\frac{v}{|cz+d|^2+|c|^2 v^2} \bigg),
\quad w=(z,v).
\end{equation*}
The action of $\operatorname{SL}_2(\mathbb{C})$ on
$\mathbb{H}^3$ is transitive, and the stabiliser of a point is $\operatorname{SU}_2(\mathbb{C})$.
The Laplace operator
$\Delta:=v^2( \partial^2/\partial x^2+
\partial^2/\partial y^2+\partial^2/\partial v^2)-v \partial/\partial v$
acts on $C^{\infty}(\mathbb{H}^3)$
and commutes with the action of $\operatorname{SL}_2(\mathbb{C})$ on $C^{\infty}(\mathbb{H}^3)$.

\subsection{Congruence subgroup $\Gamma_1(\lambda^4)$ and Kubota multiplier} \label{multsec}
Consider the congruence subgroup
\begin{equation*}
\Gamma_1(\lambda^4):= \{ \gamma \in \operatorname{SL}_2(\mathbb{Z}[i]): \gamma \equiv I \ppmod{\lambda^4} \}.
\end{equation*}
It has finite volume with respect to the $\operatorname{SL}_2(\mathbb{C})$-invariant
Haar measure $v^{-3} dx dy dv$ on $\mathbb{H}^3$. The Kubota multiplier 
$\chi: \Gamma_1(\lambda^4) \rightarrow \{\pm 1, \pm i\}$ is defined by
\begin{equation*}
\chi(\gamma)=\begin{cases}
\big( \frac{c}{a} \big)_4  & \text{if } c \neq 0, \\
1 & \text{if } c=0.
\end{cases}
\end{equation*}
Biquadratic reciprocity \eqref{bilaw} implies that $\chi$ is a character on $\Gamma_1(\lambda^4)$.
Let $\kappa \in P(\Gamma_1(\lambda^4)) \subset \mathbb{Q}(i) \cup \{\infty\}$ be a cusp of $\Gamma_1(\lambda^4)$,
and let $\Gamma_{\kappa}$ denote the stabiliser of $\kappa$
i.e. 
\begin{equation*}
\Gamma_{\kappa}=\{ \gamma \in \Gamma_1(\lambda^4) : \gamma \kappa=\kappa \}.
\end{equation*}
A cusp $\kappa \in P(\Gamma_1(\lambda^4))$ is called essential if the restriction of $\chi$ to $\Gamma_{\kappa}$ is trivial.
For each cusp $\kappa=\alpha/\gamma$ (where $\infty = 1/0$ and $\gamma=\pmatrix a b c d$), choose a scaling matrix $\sigma_{\kappa}=
\pmatrix \alpha \beta \gamma \delta \in \operatorname{SL}_2(\mathbb{Z}[i])$ 
i.e. $\sigma_{\kappa}(\infty)=\alpha/\gamma$. We have $\Gamma_{\kappa}=\sigma_{\kappa} \Gamma_{\infty} \sigma_{\kappa}^{-1}$,
where
\begin{equation*}
\Gamma_{\infty}= \big \{ \pmatrix 1 \mu 0 1 : \mu \in \lambda^{4} \mathbb{Z}[i] \big \}.
\end{equation*}
A complete set of inequivalent essential cusps \cite[Prop.~1]{Suz1} consists of 
\begin{equation} \label{cusps}
\left \{
\begin{array}{cccccccc}
0, & \frac{\lambda^4}{1+\lambda^3}, & & & & & & \\ [5pt]
\lambda^2, & i \lambda^2, & \frac{\lambda^2}{1+\lambda^3}, & \frac{i \lambda^2}{1+\lambda^3}, & \lambda^3, & \frac{\lambda^3}{1+\lambda^3}, & & \\ [5pt]
1, & \frac{1}{-1}, & 1+\lambda^3, & \frac{1+\lambda^3}{-1}, & \frac{1}{1+\lambda^3}, & \frac{1}{-1-\lambda^3}, & \frac{1+\lambda^3}{1-\lambda^3}, & \frac{1+\lambda^3}{-1+\lambda^3} \\ 
[5pt]
\infty, & \frac{1+\lambda^3}{\lambda^4}, \\ [5pt]
\frac{1}{\lambda^2}, & \frac{1+\lambda^3}{\lambda^2}, & \frac{1}{i \lambda^2}, & \frac{1+\lambda^3}{i \lambda^2} & \frac{1}{\lambda^3}, & \frac{1+\lambda^3}{\lambda^3}
\end{array}.
\right.
\end{equation}

\subsection{Metaplectic Eisenstein series} \label{metaeissec}
Here we follow \cite[\S2-\S3]{Pat1} and \cite[\S 2]{Dia}.
For any $2 \times 2$ matrix $M=(m_{ij})$, consider its $\ell$-fold tensor power
$M_\ell$. Let $\vec{i}=(i_1,\ldots,i_{\ell}) \in \{1,2\}^{\ell}$ and $\vec{j}=(j_1,\ldots,j_{\ell}) \in \{1,2\}^{\ell}$ be multi-indices.
The entries of the $2^{\ell} \times 2^{\ell}$ matrix $M_{\ell}:=M^{\otimes \ell}$ are given by 
\begin{equation*}
m_{\vec{i} \hspace{0.025cm} \vec{j}}^{(\ell)}=\prod_{r=1}^{\ell} m_{i_r j_r}.
\end{equation*}
Set $M_0:=1$. We also have $(M^{*})_{\ell}=(M_{\ell})^{*}$, where $^{*}$ denotes the conjugate transpose of the matrix,
and we write $M^{*}_{\ell}$ for both. Also note that $(MN)_{\ell}=M_{\ell} N_{\ell}$.

For $\gamma=\pmatrix a b c d \in \operatorname{SL}_2(\mathbb{C})$ 
and $w \in \mathbb{H}^3$ given in matrix form, let 
\begin{equation*}
j(\gamma,w):=\frac{\widetilde{c}w+\widetilde{d}}{\text{det}(\widetilde{c}w+\widetilde{d})^{1/2}}.
\end{equation*}
We then have the relations
\begin{equation*}
j(\gamma \gamma^{\prime},w)=j(\gamma,\gamma^{\prime}w) j(\gamma^{\prime},w),
\end{equation*}
and
\begin{equation*}
j(\gamma,w)^{*}=j(\gamma,w)^{-1}.
\end{equation*}
We write $\vec{1}$ and $\vec{2}$ for the multiindices $(1,1,\ldots,1)$
and $(2,2,\ldots,2)$ respectively.  We write $j_{\ell}(\gamma,w)$ (resp. $j^{*}_{\ell}(\gamma,w)$)
for $j(\gamma,w)_{\ell}$ (resp. $j(\gamma,w)_{\ell}^{*}$).
 
Let $P^{\prime}(\Gamma_1(\lambda^4))=\{\kappa_1,\ldots,\kappa_{24} \}$ be the complete set of inequivalent 
essential cusps of $\Gamma_1(\lambda^4)$ given in \eqref{cusps}.
Let $\kappa_1=\infty=1/0$ and write $\kappa_i=\alpha_i/\gamma_i$ where $\alpha_i$ and $\gamma_i$ are exactly
as given in \eqref{cusps}. We
take $\gamma_i=1$ if $\kappa_i=0$.
Set $\Gamma_i:=\Gamma_{\kappa_i}$ and $\sigma_{i}:=\sigma_{\kappa_i}$.
The Eisenstein series on $\Gamma_1(\lambda^4)$ attached to the
essential cusp $\kappa_i \in P^{\prime}(\Gamma_1(\lambda^4))$ is given by 
\begin{equation*}
E^{(4)}_{i}(w,s,\ell):=\sum_{\gamma \in \Gamma_{i} \backslash \Gamma_1(\lambda^4) } \overline{\chi(\gamma)}
j^{*}_{\ell}(\sigma^{-1}_i \gamma,w) v(\sigma_{i}^{-1} \gamma w)^s, \quad w=(z(w), v(w)) \in \mathbb{H}^3, \quad \sigma>2.
\end{equation*}
The series on the right side converges absolutely for $\sigma>2$.
The Eisenstein series satisfy the transformation law
\begin{equation*}
E^{(4)}_{i}(\gamma w,s,\ell)=\chi(\gamma) j_{\ell}(\gamma,w) E^{(4)}_{i}(w,s,\ell), \quad \text{for all} \quad \gamma \in \Gamma_1(\lambda^4),
\end{equation*}
and each $E_i(w,s,0)$ is an eigenfunction of the Laplacian $\Delta$.
For $u \in \mathbb{C}$ let 
\begin{equation} \label{Kintdef}
K(u,s,\ell):=\int_{\mathbb{C}} (|z|^2+1)^{-s-\ell/2} \pMatrix z {-1} {1} {\overline{z}}_{\ell}^{*} \check{e}(-uz) dx dy. 
\end{equation}
In particular, 
\begin{equation} \label{KBesseleval}
K(u,s,0)=(2 \pi)^s |u|^{s-1} \Gamma(s)^{-1} K_{s-1}(4 \pi |u|),
\end{equation}
where $K_{\xi}(\cdot)$ is the standard $K$-Bessel function of order $\xi \in \mathbb{C}$.

The Eisenstein series $E_i^{(4)}$ has the Fourier expansion at essential cusp $\kappa_j \in P^{\prime}(\Gamma_1(\lambda^4))$,
\begin{align} \label{eisfourier}
E^{(4)}_{i}(\sigma_j w,s,\ell)&=\delta_{ij} I_{\ell} v(w)^s+v(w)^{2-s} K(0,s,\ell) \phi^{(4)}_{ij}(s,0,\ell)  \nonumber \\
&+\sum_{\substack{\nu \in \lambda^{-2} \mathbb{Z}[i] \\ \nu \neq 0 }} v(w)^{2-s} K \Big( \frac{\nu v(w)}{\lambda^4},s,\ell \Big) \phi^{(4)}_{ij}(s,\nu,\ell) \check{e} \Big ( \frac{\nu z(w)}{\lambda^4} \Big),
\end{align}
where for $\nu \in \lambda^{-2} \mathbb{Z}[i]$ ($\nu=0$ is allowed in this definition) we have
\begin{align} \label{phiexp}
\phi^{(4)}_{ij}(s,\nu,\ell):=\mathcal{V}^{-1} \sum_{\substack{\gamma \in \Gamma_i \backslash \Gamma / \Gamma_j \\ c_{\gamma} \neq 0 }} 
\overline{\chi(\gamma)} \check{e} \Big( \frac{\nu d_{\gamma}}{\lambda^4 c_{\gamma}} \Big) (\widetilde{c}_{\gamma})^{*}_{\ell} & N(c_{\gamma})^{-s-\ell/2},  \quad \sigma>2,   \nonumber
\end{align}
and $\gamma=\sigma_i \pmatrix * *{ c_{\gamma}} {d_{\gamma}} \sigma^{-1}_j$,
and $\mathcal{V}$ denotes the volume of $\mathbb{C}/\lambda^4 \mathbb{Z}[i]$ with respect to $dxdy$.
Observe that each $\phi^{(4)}_{ij}(s,\nu,\ell)$ is a diagonal $2^{\ell} \times 2^{\ell}$-matrix
for $1 \leq i,j \leq 24$ integers.  Explicit formulae for the entries of $K(0,s,\ell)$ in \eqref{Kintdef}
are given in \cite[(2.28)--(2.29)]{Dia}.

Let
\begin{equation} \label{def-zeta-lambda}
\zeta_{\mathbb{Q}(i),\lambda}(s,\ell):=\sum_{m \equiv 1 \ppmod{\lambda^{3}}} \frac{\big( \frac{\overline{m}}{|m|} \big)^{4 \ell}}{N(m)^s}, \quad \sigma>1,
\end{equation}
be the Hecke $L$-function (with local factor at $\lambda$ removed) associated to the Gr\"{o}ssencharakter $\big( \frac{\overline{m}}{|m|} \big)^{4 \ell}$.

The meromorphic continuation and functional equations for the 
$E_i^{(4)}(w,s,\ell)$
and their Fourier-Whittaker coefficients $\phi^{(4)}_{ij}(s,\nu,\ell)$ follow as in Selberg \cite{Sel}.
Many analytic properties of Eisenstein series
are determined by their constant terms
using the Maass-Selberg formula \cite{Sel,Pat1}. Both 
$E^{(4)}_{i}(w,s,\ell)$ and $\phi^{(4)}_{ij}(s,\nu,\ell)$ have meromorphic continuation to
all of $\mathbb{C}$, and the poles of $E^{(4)}_{i}(w,s,\ell)$ are amongst those of
$K(0,s,\ell) \phi^{(4)}_{ij}(s,0,\ell)$ (see \eqref{eisfourier}) for $j=1,\ldots,24$. These Eisenstein series 
also satisfy a vector functional equation.
Set
\begin{equation} \label{topvec}
\boldsymbol{E}^{(4)}(w,s,\ell):=(E_i^{(4)}(w,s,\ell))_i,
\end{equation}
and 
\begin{equation} \label{Phidef}
\boldsymbol{\Phi}^{(4)}(s,\ell):=(K(0,s,\ell) \phi^{(4)}_{ji}(s,0,\ell)).
\end{equation}
The object in \eqref{topvec} (resp. \eqref{Phidef})
is a $1 \times 24$ row vector (resp. $24 \times 24$ matrix), with each entry in both being a $2^{\ell} \times 2^{\ell}$ matrix.
We have the functional equation
\begin{equation} \label{Efunceq}
\boldsymbol{E}^{(4)}(w,s,\ell)=\boldsymbol{E}^{(4)}(w,2-s,\ell) \cdot \boldsymbol{\Phi}^{(4)}(s,\ell),
\end{equation}
where 
\begin{equation*}
\boldsymbol{\Phi}^{(4)}(s,\ell) \boldsymbol{\Phi}^{(4)}(2-s,\ell)=I.
\end{equation*}

Suzuki \cite[Prop.~2]{Suz1} and Diaconu \cite[pg.~650-652]{Dia} relate elements of
the matrix of functions $\phi^{(4)}_{ij}(s,\nu,\ell)$ to the Dirichlet series $\psi^{(4)}_{\beta}(s,\nu,\ell)$
from \eqref{dir}. Recall that $\kappa_1=\infty=1/0$, and for any $0 \neq \nu \in \lambda^{-2} \mathbb{Z}[i]$ and 
$i=1,\ldots,24$, let
\begin{equation} \label{psidoubleindex}
\psi^{(4)}_{i1}(s,\nu,\ell):=\begin{cases}
\phi^{(4)}_{i1}(s,\nu,\ell)_{\vec{1} \hspace{0.025cm} \vec{1}}  & \text{if} \quad \ell>0, \\
\phi^{(4)}_{i1}(s,\nu,-\ell)_{\vec{2} \hspace{0.025cm} \vec{2}} & \text{if} \quad \ell<0, \\
\phi^{(4)}_{i1}(s,\nu,0) & \text{if} \quad \ell=0.
\end{cases}
\end{equation} \

Let $1 \leq i_1,i_2 \leq 24$ be such that 
$\kappa_{i_1}=0=0/1$ and $\kappa_{i_2}=\tfrac{\lambda^4}{1+\lambda^3}$. Let $\nu \in \lambda^{-2} \mathbb{Z}[i]$.
Then by \cite[(2.37)]{Dia},
\begin{align}
\mathcal{V} \psi^{(4)}_{i_1 1}(s,\nu,\ell)&=(-1)^{\ell} \psi^{(4)}_{1}(s,\nu,\ell); \label{connect1} \\
\mathcal{V} \psi^{(4)}_{i_2 1}(s,\nu,\ell)&=(-1)^{\ell+1} \psi^{(4)}_{1+\lambda^3}(s,\nu,\ell), \label{connect2}
\end{align}
where the functions $\psi_{\beta}^{(4)}(s,\nu,\ell)$ on the right side of the above displays are given by 
\eqref{dir}.
This yields the meromorphic continuation of the $\psi^{(4)}_{\beta}(s,\nu,\ell)$
to $\mathbb{C}$. 

\begin{remark} \label{lincombo}
For each $i=1,\ldots,24,$ and $0 \neq \nu \in \lambda^{-2} \mathbb{Z}[i]$,
the function $\mathcal{V} \psi^{(4)}_{i 1}(s,\nu,\ell)$
is equal to
a finite $\mathbb{C}$-linear combination
of $\psi^{(4)}_{1}(s,(-1)^{a} \lambda^{2b} \nu,\ell)$ and
$\psi^{(4)}_{1+\lambda^3}(s,(-1)^{a} \lambda^{2b} \nu,\ell)$ with $a,b \in \mathbb{Z}_{\geq 0}$.
There are two cases  according to
whether the essential cusp
\begin{equation} \label{cuspcases}
\kappa_i=\alpha_i/\gamma_i \quad \text{has} \quad (\gamma_i,\lambda)=1 \quad \text{or} \quad
(\gamma_i,\lambda) \neq 1. 
\end{equation}
If $(\gamma_i,\lambda)=1$, 
then by \eqref{cusps} we have $\gamma_i \in \{ \pm 1, \pm (1+\lambda^3), \pm (1-\lambda^3)\}$.
Let $\gamma_i^{\prime} \in \mathbb{Z}[i]$
be such that $\gamma_i \gamma_i^{\prime} \equiv 1 \ppmod{4}$.
Then 
by \cite[(2.35)--(2.36)]{Dia} and the argument following those equations,
we have 
\begin{equation} \label{coprimecusp}
\mathcal{V} \psi^{(4)}_{i 1}(s,\nu,\ell)=\overline{\Big( \frac{\alpha_i}{\gamma_i} \Big)}_4 \check{e}
\Big({-\frac{\gamma_i^{\prime} \alpha_i \nu}{\lambda^4}} \Big) \cdot
\begin{cases}
\psi^{(4)}_{1}(s,\nu,\ell) & \text{if } \gamma_i=-1 \\
\psi^{(4)}_{1+\lambda^3}(s,\nu,\ell) & \text{if } \gamma_i=-1-\lambda^3,-1+\lambda^3 \\
(-1)^{\ell} \psi^{(4)}_{1}(s,\nu,\ell)  &  \text{if } \gamma_i =1  \\
(-1)^{\ell+1} \psi^{(4)}_{1+\lambda^3}(s,\nu,\ell) & \text{if } \gamma_i=1+\lambda^3,1-\lambda^3.
\end{cases}
\end{equation}
If $(\gamma_i,\lambda) \neq 1$, then by \eqref{cusps} we have
$\gamma_i \in \{0,\lambda^4,\lambda^2,i \lambda^2, \lambda^3 \}$
and $\alpha_i \in \{1,1+\lambda^3\}$.  By
\cite[(2.38)--(2.40)]{Dia}) and the argument following those equations,
we have
\begin{align} \label{psilinearcombo}
\mathcal{V} \psi^{(4)}_{i 1}(s,\nu,\ell) &= \overline{\Big(\frac{-\gamma_i}{\alpha_i} \Big) }_4
\sum_{(a,b) \in S(\kappa_i,\nu)} (-i)^{\ell a}  \Big( \frac{{\overline{\lambda}}^{b}}{|\lambda^{b}|} \Big)^{\ell} 
2^{-bs}  \nonumber \\
& \times \Big(\psi^{(4)}_1(s,(-1)^{a} \lambda^{2b} \nu,\ell) \cdot \begin{cases}
\Gamma_1(\nu,i^{a} \lambda^b) & \text{if } \alpha_i=1 \ \\
\Gamma_{1+\lambda^3}(\nu,i^{a} \lambda^b) & \text{if } \alpha_i=1+\lambda^3
\end{cases} \nonumber \\
&+\psi^{(4)}_{1+\lambda^3}(s,(-1)^{a} \lambda^{2b} \nu,\ell) \cdot 
\begin{cases}
\Gamma_{1+\lambda^3}(\nu,i^{a} \lambda^b) & \text{if } \alpha_i=1 \\
-\Gamma_1(\nu,i^{a} \lambda^b)  & \text{if } \alpha_i=1+\lambda^3
\end{cases} \Big),
\end{align}
where 
\begin{equation} \label{gammagauss}
\Gamma_{\beta}(\nu, i^{a} \lambda^b):=\sum_{\substack{ h \ppmod{\lambda^{b+4}} \\ h \equiv \beta \ppmod{4} }}
\Big( \frac{i^a \lambda^b }{h} \Big)_4 \check{e} \Big( \frac{h \nu}{i^a \lambda^{b+4}} \Big),
\end{equation}
and 
\begin{equation} \label{suppdef}
S(\kappa,\nu):=\{(a,b): \Gamma_1 (\nu,i^a \lambda^b) \neq 0 \} \cup \{(a,b): \Gamma_{1+\lambda^3}(\nu,i^a \lambda^b) \neq 0 \}
\end{equation}
is a subset of $\{0,1,2,3\} \times \mathbb{Z}_{\geq 2}$, that happens to be finite. We now give more details.

If $0 \neq \nu=i^j \lambda^k \nu^{\prime} \in \mathbb{Z}[i]$ with $0 \leq j \leq 3$, $k \in \mathbb{Z}_{\geq -2}$, and
$\nu^{\prime} \equiv 1 \ppmod{\lambda^3}$, then
\begin{align} \label{gammatransform}
\Gamma_{\beta}(i^j \lambda^k \nu^{\prime}, i^{a} \lambda^b)=
\overline{\Big( \frac{i^a \lambda^b}{\nu^{\prime}} \Big)}_4
\cdot \begin{cases}
\Gamma_{\beta}(i^j \lambda^k, i^{a} \lambda^b) & \text{if } \nu^{\prime} \equiv 1 \ppmod{4} \\
\Gamma_{\beta(1+\lambda^3)}(i^j \lambda^k, i^{a} \lambda^b) & \text{if } \nu^{\prime} \equiv 1+\lambda^3 \ppmod{4}
\end{cases}.
\end{align}
The Gauss sums on the right side of \eqref{gammatransform} are explicitly computed in \cite[pg.~81-83]{Suz1},
but their exact values are not important for this paper. We outline the important properties now.
The equalities \eqref{suppdef} and \eqref{gammatransform} imply that the set indexing the sum in \eqref{psilinearcombo} satisfies 
\begin{equation} \label{indexprop}
S(\kappa,i^j \lambda^k \nu^{\prime})=S(\kappa,i^j \lambda^k).
\end{equation}
Moreover, by \cite[pg.~82]{Suz1} we have that 
for given $0 \leq j \leq 3$ and $k \in \mathbb{Z}_{\geq -2}$,
\begin{align} \label{gammasupport}
\Gamma_{\beta}(i^j \lambda^k, i^{a} \lambda^b) \neq 0 \implies 2 \leq b \leq k+5 \quad \text{and} \quad 0 \leq a \leq 3.
\end{align}
Thus \eqref{indexprop} and \eqref{gammasupport} imply that
for all $0 \neq \nu \in \lambda^{-2} \mathbb{Z}[i]$,
\begin{equation} \label{gammauniformbd}
|\Gamma_{\beta}(\nu, i^{a} \lambda^b)| \ll_{\text{ord}_{\lambda}(\nu)} 1,
\end{equation}
uniformly for all $(a,b) \in \{0,1,2,3\} \times \mathbb{Z}_{\geq 2}$. 
To keep the notation consistent between the two families of cusps in \eqref{cuspcases},
we set 
\begin{equation} \label{Scoprimecusp}
S(\kappa_i,\nu):=\{(0,0)\} \text{ for all } 0 \neq \nu \in \lambda^{-2} \mathbb{Z}[i]
 \text{ and } \kappa_i=\alpha_i/\gamma_i \text{ such that }
(\gamma_i,\lambda)=1,
\end{equation}
see \eqref{coprimecusp}.
\end{remark}

By a similar argument, it can be shown that  
up to a rational function 
of $2^{-s}$ depending on $\ell$, $i$, $j$,
the functions $\phi^{(4)}_{ji}(s,0,\ell)_{\vec{1} \hspace{0.025cm} \vec{1}}$ are equal to 
$\zeta_{\mathbb{Q}(i),\lambda}(4s-4,\ell) \zeta_{\mathbb{Q}(i),\lambda}(4s-3,\ell)^{-1}$.
Recall that $\kappa_1=\infty=1/0$, and
that $1 \leq i_1,i_2 \leq 24$ are such that 
$\kappa_{i_1}=0$ and $\kappa_{i_2}=\tfrac{\lambda^4}{1+\lambda^3}$. 
For instance, from \cite{Suz1} and \cite[(2.43)--(2.44)]{Dia} we have that 
\begin{align*}
\mathcal{V} \phi^{(4)}_{ji_1}(s,0,0)_{\vec{1} \hspace{0.025cm} \vec{1}}=
\begin{cases}
2(2^{4s-4}-1)^{-1} \zeta_{\mathbb{Q}(i),\lambda}(4s-4,0) \zeta_{\mathbb{Q}(i),\lambda}(4s-3,0)^{-1} & \text{if } j=i_1 \\
\zeta_{\mathbb{Q}(i),\lambda}(4s-4,0) \zeta_{\mathbb{Q}(i),\lambda}(4s-3,0)^{-1} & \text{if } \alpha_j=1 \\
0 & \text{otherwise}
\end{cases},
\end{align*}
and
\begin{align*}
\mathcal{V} \phi^{(4)}_{ji_2}(s,0,0)_{\vec{1} \hspace{0.025cm} \vec{1}}=
\begin{cases}
2(2^{4s-4}-1)^{-1} \zeta_{\mathbb{Q}(i),\lambda}(4s-4,0) \zeta_{\mathbb{Q}(i),\lambda}(4s-3,0)^{-1} & \text{if } j=i_2 \\
\overline{\big( \frac{-\gamma_j}{\alpha_j} \big)_4} \zeta_{\mathbb{Q}(i),\lambda}(4s-4,0) \zeta_{\mathbb{Q}(i),\lambda}(4s-3,0)^{-1} & \text{if } \alpha_j=1+\lambda^3 \\
0 & \text{otherwise}
\end{cases}.
\end{align*}

Using \cite[(2.25)--(2.33)]{Dia}
and the above description of the $\phi^{(4)}_{ji}(s,0,\ell)_{\vec{1} \hspace{0.025cm} \vec{1}}$ one can elucidate the 
functional equation \eqref{Efunceq}. To this end, for $0 \neq \nu \in \lambda^{-2} \mathbb{Z}[i]$ let 
\begin{align} 
Z_{i1}(s,\nu,\ell)&:=\zeta_{\mathbb{Q}(i),\lambda}(4s-3,\ell) \psi^{(4)}_{i1}(s,\nu,\ell);  \label{Zi1defn}  \\
\widehat{\psi}^{(4)}_{i1}(s,\nu,\ell)&:=G_{\infty}(s,\ell) Z_{i1}(s,\nu,\ell),  \label{hatdef}
\end{align}
where 
\begin{equation} \label{Ginfty}
G_{\infty}(s,\ell):=\Gamma_{\mathbb{C}} \Big(s+\frac{|\ell|}{2}-\frac{3}{4} \Big) \Gamma_{\mathbb{C}} \Big(s+\frac{|\ell|}{2}-\frac{1}{2} \Big) \Gamma_{\mathbb{C}} \Big(s+\frac{|\ell|}{2}-\frac{1}{4} \Big),
\end{equation}
and 
\begin{equation} \label{gammaC}
\Gamma_{\mathbb{C}}(s):=2 \cdot (2 \pi)^{-s} \Gamma(s).
\end{equation}
The functional equation in \cite[(2.48)]{Dia} reads
\begin{equation} \label{funceqn}
\widehat{\psi}^{(4)}_{i1}(s,\nu,\ell)=N(\nu)^{1-s} \Big( \frac{\overline{\nu}}{|\nu|} \Big)^{\ell} 
\cdot \sum_{j=1}^{24} A_{ji}(2^{-s},\ell) \widehat{\psi}^{(4)}_{j1}(2-s,\nu,-\ell),
\end{equation}
where $A_{ji}(2^{-z},\ell)$ is a rational function in $2^{-z}$ with coefficients in $\mathbb{C}$ and $0 \neq \nu \in \lambda^{-2} \mathbb{Z}[i]$.
The rational functions 
$A_{ji}(2^{-z},\ell)$ are holomorphic for $\Re(z) \neq 1, 5/4$. The analogous rational functions in the cubic case
can be deduced from \cite[Theorem~6.1]{Pat1}.

\begin{prop} \label{convexity}
The functions $\widehat{\psi}^{(4)}_{i1}(s,\nu,\ell)$ for $\ell \in \mathbb{Z}$,  $0 \neq \nu \in \lambda^{-2} \mathbb{Z}[i]$,
and $i=1,\ldots,24,$ can be meromorphically continued to $\mathbb{C}$. If $\ell \neq 0$ they are entire,
and if $\ell=0$ they have at most two possible (simple) poles
at $s=3/4$ and $s=5/4$. In all cases these functions are bounded when $\Im(s)$ is large in vertical strips of finite width,
and satisfy the functional equation \eqref{funceqn}.

For $0<\varepsilon<1/100$ we have,
\begin{align} \label{convexbd}
\psi^{(4)}_{i1}(s,\nu,\ell) & \ll_{\varepsilon, \emph{ord}_{\lambda}(\nu)} N(\nu)^{(1/2)(3/2-\sigma)+\varepsilon} (|s|^2+\ell^2+1)^{(3/2)(3/2-\sigma)+\varepsilon}; \\
& \quad \text{for} \quad 1+\varepsilon<\sigma<3/2+\varepsilon/2 \quad \text{and} \quad  |s- 5/4 |>1/8. \nonumber 
\end{align}
\end{prop}

\begin{proof}
The claims regarding the meromorphic continuation to all of $\mathbb{C}$, the location of possible poles, and the boundedness 
of $\widehat{\psi}^{(4)}_{i1}(s,\nu,\ell)$ in fixed vertical strips all follow from \cite[Theorem~2.1]{Dia}. 

The proof of the convexity bound \eqref{convexbd} is analogous to the argument on \cite[pg.~127]{HBP} 
with some straightforward modifications. For $\beta \in \{1,1+\lambda^3 \} \ppmod{4}$ recall the definition \eqref{dir}.
By \eqref{rel1}--\eqref{rel2}, \eqref{rel4},
and \eqref{sqrootcancel}, we have for $\sigma>3/2+\varepsilon$,
\begin{align} \label{absconvbd}
|\psi^{(4)}_{\beta}(s,\nu,\ell)| & \leq \prod_{\substack{ \pi^k \mid \! \mid \nu \\ \pi \equiv 1 \ppmod{\lambda^3} \\ k \in \mathbb{Z}_{\geq 0} \\
k \equiv 0,1,2 \ppmod{4}}}(1+N(\pi)^{1/2+k-(k+1) \sigma}) 
\prod_{\substack{\pi^{k} \mid \! \mid \nu \\ \pi \equiv 1 \ppmod{\lambda^3} \\ k \equiv 3 \ppmod{4}}} (1+N(\pi)^{k-(k+1) \sigma} ) \nonumber \\
 & \times \prod_{\substack{\pi^k \mid \! \mid \nu \\ \pi \equiv 1 \ppmod{\lambda^3} }} \Big(\sum_{\substack{0 \leq m \leq k \\ m \equiv 0 \ppmod{4}}} 
 N(\pi)^{m(1-\sigma)} \Big) \nonumber \\
 & \ll 1.
\end{align}
Thus from Remark \ref{lincombo} (see \eqref{coprimecusp}--\eqref{gammauniformbd}), and \eqref{absconvbd}, we obtain
\begin{equation} \label{absconvbd2}
Z_{i1}(s,\nu,\ell) \ll_{\text{ord}_{\lambda}(\nu)} 1 \quad \text{for} \quad \sigma>\frac{3}{2}+\varepsilon,
\end{equation}
where $Z_{i1}(s,\nu,\ell)$ is given in \eqref{Zi1defn}.
Fix $0<\delta<\varepsilon/4$. From \eqref{absconvbd2}, applied with
$\varepsilon$ replaced by $\delta/2$, and \eqref{funceqn} one deduces that
\begin{equation} \label{12line1}
Z_{i1}(s,\nu,\ell) \ll_{\text{ord}_{\lambda}(\nu)} N(\nu)^{1/2+\delta} \cdot \frac{|G_{\infty}(2-s,-\ell)|}{|G_{\infty}(s,\ell)|} \quad \text{for} \quad
\sigma=1/2-\delta,
\end{equation}
for $i=1,\ldots,24$.
By Stirling's formula \cite[(5.11.1)]{NIST:DLMF} we have 
\begin{equation} \label{12line2}
\frac{|G_{\infty}(2-s,-\ell)|}{|G_{\infty}(s,\ell)|} \ll_{\delta} (|s|^2+\ell^2+1)^{3/2+3\delta} \quad \text{for} \quad
\sigma=1/2-\delta.
\end{equation}
Applying the Phragmen--Lindel\"{o}f principle to $Z_{i1}(s,\nu,\ell)$
(it has order $1$), and using the estimates 
\eqref{absconvbd2} and \eqref{12line1}--\eqref{12line2} on the lines
$\sigma=3/2+\delta$ and $\sigma=1/2-\delta$, respectively,
and then using absolute convergence to the right of this strip, we obtain
\begin{align} \label{convexprod}
Z_{i1}(s,\nu,\ell) &\ll_{\text{ord}_{\lambda}(\nu)} 
N(\nu)^{(1/2)(3/2-\sigma)+\varepsilon} (|s|^2+\ell^2+1)^{(3/2)(3/2-\sigma)+\varepsilon}, \nonumber \\
& \quad \text{for} \quad 1/2+\varepsilon<\sigma<3/2+\varepsilon/2, \quad  |s- 5/4 |, \hspace{0.05cm}
|s-3/4|>1/8,
\end{align}
and $i=1,\ldots,24$.
After noting that $|\zeta_{\mathbb{Q}(i),\lambda}(4s-3,\ell)| \gg 1$ for $\sigma>1+\varepsilon$,
we divide the bound in \eqref{convexprod} by $|\zeta_{\mathbb{Q}(i),\lambda}(4s-3,\ell)|$ 
and obtain \eqref{convexbd}.
\end{proof}

\begin{remark} \label{Zremark} 
For $i=1,\ldots,24$, Proposition \ref{convexity} tells us that $\widehat{\psi}_{i1}^{(4)}(s,\nu,0)$ from \eqref{hatdef}
has at most simple poles at $s=5/4$ and $s=3/4$.
Since $G_{\infty}(s,0)$ in \eqref{Ginfty} has a simple pole at $s=3/4$, then by
\eqref{hatdef} we have that $Z_{i1}(s,\nu,0)$ is analytic at $s=3/4$.
\end{remark}

\subsection{Theta functions} \label{thetasec}
Kubota \cite{Kub1,Kub2} introduced the notion of theta functions as residues of
metaplectic Eisenstein series on the $n$-fold cover of $\operatorname{GL}_2$ for $n \geq 3$. 
Kazhdan and Patterson \cite{KP} further studied metaplectic theta functions coming from the $n$-fold cover of 
$\operatorname{GL}_r$. The reader may consult the beautiful surveys 
\cite{Hoff,CFH} for a summary of this topic in the classical language.

On the $4$-fold cover of $\operatorname{GL}_2$ we have the quartic theta functions,
\begin{equation*}
\vartheta^{(4)}_i(w):=\text{Res}_{s=5/4} E_{i}^{(4)}(w,s,0), \quad w \in \mathbb{H}^3, \quad i=1,\ldots,24,
\end{equation*}
and from \eqref{KBesseleval}--\eqref{eisfourier}
we see that they have Fourier expansions,
\begin{equation} \label{thetafourier}
\vartheta^{(4)}_i(w)=\tau^{(4)}_i(\infty)v(w)^{3/4}+c \sum_{\substack{\nu \in \lambda^{-2} \mathbb{Z}[i] \\ \nu \neq 0 }}
\tau^{(4)}_i(\nu) N(\nu)^{1/8} v(w) K_{1/4} (\pi |\nu| v(w)) \check{e} \Big( \frac{\nu z(w)}{\lambda^4} \Big),
\end{equation}
where $c=2^{3/4} \pi^{5/4} \Gamma ( \frac{5}{4})^{-1}$.
The Fourier-Whittaker coefficients of the $\vartheta^{(4)}_i(w)$  are intimately connected to the
residues $\psi^{(4)}_1(\nu)$ and $\psi^{(4)}_{1+\lambda^3}(\nu)$ from \eqref{res1}. 
If $1 \leq i_1,i_2 \leq 24$ are such that 
$\kappa_{i_1}=0$ and $\kappa_{i_2}=\tfrac{\lambda^4}{1+\lambda^3}$, 
then from \eqref{KBesseleval}--\eqref{eisfourier},
\eqref{psidoubleindex}--\eqref{connect2}, and \eqref{thetafourier}, we have
\begin{equation} \label{taurel}
\mathcal{V} \tau^{(4)}_{i_1}(\nu)= \psi^{(4)}_{1}(\nu) \quad \text{and} \quad  \mathcal{V} \tau^{(4)}_{i_2}(\nu)=
-\psi^{(4)}_{1+\lambda^3}(\nu) \quad 
\text{for} \quad 0 \neq \nu \in \lambda^{-2} \mathbb{Z}[i].
\end{equation}

Discovery of a closed form formula for the Fourier-Whittaker coefficients of theta functions 
on the $n$-fold cover of $\operatorname{GL}_2$ 
is a long standing open problem. Patterson \cite{Pat1}
settled this problem in the case of the cubic theta function (3-fold cover of $\operatorname{GL}_2$) 
using a metaplectic Hecke converse argument. Patterson proved that the Fourier-Whittaker coefficients for the 
 cubic theta function are essentially cubic Gauss sums.
Unfortunately Patterson's approach did not extend to the 
$n$-fold cover of $\operatorname{GL}_2$ for $n \geq 4$. 

Suzuki \cite{Suz1} obtained partial information
about the residues on the $4$-fold cover of $\operatorname{GL}_2$ using a different method. 
In what follows, let $m \in \mathbb{Z}[i]$ be squarefree and satisfy $m \equiv 1 \pmod{\lambda^3}$.
Let $0 \neq \nu \in \lambda^{-2} \mathbb{Z}[i]$ and $\beta \in \{1,1+\lambda^3 \} \ppmod{4}$.
Suzuki \cite[Theorem~3]{Suz1} established \emph{periodicity} for biquadrates, 
\begin{equation} \label{fourrel}
\psi^{(4)}_{\beta}(m^4 \nu)=\psi^{(4)}_{\beta}(\nu),
\end{equation}
as well as the vanishing property involving cubes \cite[Theorem~4]{Suz1},
\begin{equation} \label{cuberel}
\psi^{(4)}_{\beta}(m^3 \nu)=0 \quad \text{when} \quad m \neq 1 \quad \text{and} \quad (m,\nu)=1.
\end{equation}
Suzuki's property concerning squares \cite[Theorem~5]{Suz1} is important for this paper i.e.
\begin{equation} \label{sqrel}
\psi^{(4)}_{\beta}(m^2 \nu)=
\begin{cases}
\frac{\overline{g_4(\nu,m)}}{N(m)^{3/4}} \cdot \psi^{(4)}_{\beta}(\nu) & \text{if} \quad m \equiv 1 \ppmod{4}, \\
(-1)^{\tfrac{N(\beta)-1}{4}} \cdot \frac{\overline{g_4(\nu,m)}}{N(m)^{3/4}} \cdot \psi^{(4)}_{\beta (1+\lambda^3)}(\nu) &
\text{if} \quad m \equiv 1+\lambda^3 \ppmod{4}
\end{cases},
\end{equation}
for $(m,\nu)=1$.

\begin{corollary} \label{coeffcor}
For $m \in \mathbb{Z}[i]$ squarefree with $m \equiv 1 \ppmod{\lambda^3}$ and $0 \neq \nu \in \lambda^{-2} \mathbb{Z}[i]$
with $(m,\nu)=1$, we have 
\begin{equation*}
\psi^{(4)}_{\beta}(m^2 \nu) \ll_{\nu} N(m)^{-1/4}.
\end{equation*}
\end{corollary}
\begin{remark}
Suzuki's property \eqref{sqrel} with $\nu=1$ 
yields the closed form evaluation of Fourier-Whittaker coefficients of $\vartheta^{(4)}(w)$ indexed by $m^2$
with $m \in \mathbb{Z}[i]$ squarefree and satisfying $m \equiv 1 \ppmod{\lambda^3}$  
(see \eqref{thetafourier} and \eqref{taurel}). On the other hand, one obtains the
inferior bound $\psi^{(4)}_{\beta}(m^2) \ll N(m)^{1/4+\varepsilon}$ if only the convexity bound from
Proposition \ref{convexity} is used.
\end{remark}

Suzuki \cite[Theorem~6]{Suz1} also proved that
\begin{equation} \label{linrel1}
\psi^{(4)}_{\beta}(m \nu) =0 \quad \text{for} \quad (m,\nu)=1,
\end{equation}
unless
\begin{equation} \label{linrel2}
g_2 \Big(\frac{m \nu}{m_1},m_1 \Big)=N(m_1)^{1/2}
\quad \text{for all} \quad m_1 \text{ s.t. } m_1 \mid m \text{ with }
m_1 \equiv 1 \ppmod{\lambda^3}.
\end{equation} 
The success of Patterson in determining the Fourier-Whittaker coefficients of the theta function on the 3-fold cover of $\operatorname{GL}_2$, 
and Suzuki the partial information \eqref{fourrel}--\eqref{linrel2} on the 4-fold cover of $\operatorname{GL}_2$, 
can be explained using periodicity and 
Hecke relations for the coefficients of theta functions.
See \cite{BH}, \cite[\S 2]{Hoff}, and \cite[\S 2]{CFH} for a summary. The properties \eqref{fourrel}--\eqref{linrel2}
say nothing about $\psi^{(4)}_{\beta}(\pi)$ for $\pi \equiv 1 \ppmod{4}$ prime for instance.
Deligne \cite{Deli} studied the problem from a representation
theoretic point of view and explained that 
the inaccessibility of the cases $n \geq 4$ on $\operatorname{GL}_2$ was owed to the fact that
local representations do not have unique Whittaker models.
Kazhdan and Patterson \cite{KP} generalised Deligne's work and showed the periodicity property held
in greater generality. 

Patterson \cite{Pat2}, with refinements by Eckhardt-Patterson \cite[Conjecture~2.11]{EckPat},
conjectured essentially that $\tau^{(4)}_{i}(\pi)^2 N(\pi)^{1/4}$ is proportional
to $\overline{\widetilde{g}_4(\pi)}$ for $\pi \equiv 1 \ppmod{4}$ prime.
Patterson's quartic conjecture was based on a conjectural identity of his involving 
two Dirichlet series that don't have an Euler product, see \cite[(3)]{BH} and \cite[pg.~154]{CFH}.
Patterson's quartic conjecture remains wide open,
and has been shown 
to hold on average by Bump and Hoffstein \cite[pg.~6-8]{BH}. It has been conjecturally
generalised to the $6$-fold cover of $\operatorname{GL}_2$
by Chinta, Friedberg and Hoffstein \cite[Conjecture~1]{CFH},
and numerically tested by Br\"{o}ker and Hoffstein \cite{BrHo}.

\section{Onodera's and Goldmakher--Louvel's quadratic large sieve over $\mathbb{Q}(i)$} \label{quadsievesec}
We record here Onodera's quadratic large sieve over $\mathbb{Q}(i)$ \cite{On}.
Onodera's result was a mild generalisation of Heath-Brown's famous
quadratic large sieve over $\mathbb{Q}$ \cite{HB1}.
Goldmakher and Louvel \cite{GL} subsequently generalised these results
to more general number fields.

\begin{theorem}{\emph{\cite[Theorem~1]{On}, \cite[Corollary~1.2]{GL} }} \label{quadsieve}
Let $\{a_n\}$ be an arbitrary complex sequence, 
$M,N \geq 1$, and $\varepsilon>0$. Then we have
\begin{align*}
\sum_{\substack{m \in \mathbb{Z}[i] \\ N(m) \leq M \\ m \equiv 1 \ppmod{2} }} \mu^2(m) \Big | \hspace{0.1cm}
\sum_{\substack{n \in \mathbb{Z}[i] \\  N(n) \leq N \\ n \equiv 1 \ppmod{2}}} a_n \mu^2(n) \Big( \frac{n}{m} \Big)_2 \Big |^2 \ll_{\epsilon} (MN)^\varepsilon (M+N)   
\sum_{N(n) \leq N} |a_{n}|^2 \mu^2(n).
\end{align*}
\end{theorem}

\section{Preliminary reductions and Vaughan's identity} \label{vauidsec}
Let $R:(0,\infty) \rightarrow \mathbb{C}$ be a smooth function with compact support in $[1,2]$.
We suppress the dependence on $R$ from the notation.
For $\alpha \in \mathbb{Z}[i]$ with $\alpha \equiv 1 \ppmod{\lambda^3}$, $\ell \in \mathbb{Z}$,
and $\beta \in \{1,1+\lambda^3 \} \ppmod{4}$, let 
\begin{equation} \label{Helldefn}
H_{\beta}(X,\ell):=\sum_{\substack{c \in \mathbb{Z}[i] \\ c \equiv \beta \ppmod{4} }}
\widetilde{g}_4(c) \Big(\frac{\overline{c}}{|c|} \Big)^{\ell} \Lambda(c) R \Big( \frac{N(c)}{X} \Big),
\end{equation}
and
\begin{equation} \label{Felldefn}
F_{\beta}(X,\ell; \alpha):=\sum_{\substack{c \in \mathbb{Z}[i] \\ c \equiv \beta \ppmod{4} \\ c \equiv 0 \ppmod{\alpha} }}
\widetilde{g}_4(c) \Big( \frac{\overline{c}}{|c|} \Big)^{\ell} R \Big( \frac{N(c)}{X} \Big).
\end{equation}

We now record Vaughan's identity \cite{Vau}.
Our setting is analogous to \cite[\S2]{HB}. For $1 \leq u \leq X$, $j \in \{0,1,2^{\prime},2^{\prime \prime},3,4\}$,
$\ell \in \mathbb{Z}$, and $\beta \in \{1,1+\lambda^3 \} \ppmod{4}$, let
\begin{equation} \label{sigmajl}
\Sigma_{j,\beta} (X,\ell,u):=\sum_{\substack{a,b,c \in \mathbb{Z}[i] \\ a,b,c \equiv 1 \ppmod{\lambda^3} \\ abc \equiv \beta \ppmod{4}}}
\Lambda(a) \mu(b) \widetilde{g}_4(abc)
\Big( \frac{\overline{abc}}{|abc|} \Big)^{\ell} R \Big( \frac{N(abc)}{X} \Big),
\end{equation}
where $a,b,c \in \mathbb{Z}[i]$ are subject to the conditions: 
\begin{align*}
N(bc) & \leq u,  \quad \quad \quad \quad \quad \quad  j=0;  \\ 
N(b) & \leq u,  \quad \quad \quad \quad \quad \quad   j=1;  \\ 
N(ab) & \leq u,  \quad \quad \quad \quad \quad \quad j=2';  \\ 
N(a),N(b) \leq u &< N(ab),   \quad \quad \quad \quad j=2'';  \\ 
N(b) \leq u&< N(a), N(bc),   \quad \hspace{0.2cm} j=3;  \\ 
N(a), N(bc) & \leq u, \quad \quad \quad \quad \quad \quad j=4. 
\end{align*}
Then Vaughan's identity reads,
\begin{equation} \label{Vauid}
\Sigma_{0,\beta}(X,\ell,u) + \Sigma_{2',\beta}(X,\ell,u)  + \Sigma_{2^{\prime \prime},\beta}(X,\ell,u)+\Sigma_{3,\beta}(X,\ell,u)
=\Sigma_{1,\beta}(X,\ell,u)+\Sigma_{4,\beta}(X,\ell,u). 
\end{equation}
\begin{remark}
Note that the presence of the $\widetilde{g}_4(abc)$ in \eqref{sigmajl} implies that the summation variables $a,b,c \in \mathbb{Z}[i]$
in each $\Sigma_{j,\beta}(X,\ell,u)$ are supported on $\mu^2(abc)=1$.
\end{remark}

The following lemma is a straightforward modification of the statement and 
proof of \cite[Proposition~1]{HBP} and \cite[pg.~101-102]{HB},
where the condition $X < N(abc) \leq 2X$ is replaced by the smooth function $R(N(abc)/X)$ supported on $[1,2]$.
\begin{lemma} \label{type1init}
Let $R:(0,\infty) \rightarrow \mathbb{C}$ be a smooth function with compact support in $[1,2]$,
$\ell \in \mathbb{Z}$, $X \geq 1$, $1 \leq u<X^{1/2}$,
$\beta \in \{1,1+\lambda^3 \} \ppmod{4}$, $j \in \{0,1,2^{\prime},2^{\prime \prime},3,4\}$,
and $\Sigma_{j,\beta}(X,\ell,u)$ be as in \eqref{sigmajl}.
Let $F_{\beta,R_{\log}}$ denote \eqref{Felldefn} with $R(t)$ replaced by
$R_{\log}(t):=R(t)\log t$. Then
\begin{align}
\Sigma_{0,\beta}(X,\ell,u)&=H_{\beta}(X,\ell); \label{sigma0} \\
\Sigma_{4,\beta}(X,\ell,u)&=0; \label{sigma4} \\
 | \Sigma_{1,\beta}(X,\ell,u)  |  & \leq \sum_{\substack{ N(\alpha) \leq u \\ \alpha \equiv 1 \ppmod{\lambda^3} }} \mu^2(\alpha) \nonumber \\
 & \times \Big(\log X\,|F_{\beta}(X,\ell;\alpha)|
+|F_{\beta,R_{\log}}(X,\ell;\alpha)|\Big); \label{sigma1bd}   \\
 | \Sigma_{2^{\prime},\beta} (X,\ell,u) | 
& \leq  \log(2u)\sum_{\substack{N(\alpha) \leq u \\ \alpha \equiv 1 \ppmod{\lambda^3}}} \mu^2(\alpha) |F_{\beta}(X,\ell;\alpha)|.
 \label{sigma2primebd} 
\end{align}
\end{lemma}

\section{Type-II estimates} \label{type2sec}
In this section we prove estimates for the terms $\Sigma_{2^{\prime \prime},\beta}$
and $\Sigma_{3,\beta}$ appearing in \eqref{Vauid}. We use twisted multiplicativity for quartic Gauss sums and 
the quadratic large sieve over $\mathbb{Q}(i)$.
Our approach is analogous to the case of cubic Gauss sums \cite[pg.102-103]{HB}.

\begin{prop} \label{type2prop}
Let $R:(0,\infty) \rightarrow \mathbb{C}$ be a smooth function with compact support in $[1,2]$,
$\ell \in \mathbb{Z}$, $1 \leq u<X^{1/2}$,
and $\beta \in \{1,1+\lambda^3 \} \ppmod{4}$. For any $\varepsilon>0$ we have,
\begin{equation*}
\Sigma_{2^{\prime \prime},\beta}(X,\ell,u), \hspace{0.1cm} \Sigma_{3,\beta}(X,\ell,u) \ll_{\varepsilon,R} X^{\varepsilon}(X^{1/2} u+X u^{-1/2}),
\end{equation*}
uniformly in $\ell$.
\end{prop}

\begin{proof}
Using \eqref{quartquad}, \eqref{bilaw}, and \eqref{rel2}, we have
\begin{equation} \label{2primeintermed}
\Sigma_{2^{\prime \prime},\beta}(X,\ell,u)=\sum_{\substack{v,w \in \mathbb{Z}[i] \\ u<N(v) \leq u^2 \\  N(w) \leq 2X/u  \\ vw \equiv \beta \ppmod{4} }}
(-1)^{C(v,w)} A(v) B(w) \Big(\frac{w}{v} \Big)_2 R \Big( \frac{N(vw)}{X} \Big), 
\end{equation}
where $C(\cdot,\cdot)$ is given in \eqref{Cdef}, and
\begin{align}
A(v)&:= \delta_{v \equiv 1 \ppmod{\lambda^3}} \cdot \widetilde{g}_4(v) \Big(\frac{\overline{v}}{|v|} \Big)^{\ell} \sum_{\substack{ab=v \\ N(a), N(b) \leq u \\ a,b \equiv 1 \ppmod{\lambda^3} } } \Lambda(a) \mu(b);  \label{Adef} \\
B(w)&:= \delta_{w \equiv 1 \ppmod{\lambda^3}} \cdot \widetilde{g}_4(w) \Big(\frac{\overline{w}}{|w|} \Big)^{\ell}. \label{Bdef}
\end{align}
Similarly,
\begin{equation} \label{3primeintermed}
\Sigma_{3,\beta}(X,\ell,u)=\sum_{\substack{v,w \in \mathbb{Z}[i] \\ u<N(v),N(w) \leq 2X/u \\ vw \equiv \beta \ppmod{4} }}
(-1)^{C(v,w)} G(v) H(w) \Big(\frac{w}{v} \Big)_2 R \Big( \frac{N(vw)}{X} \Big), 
\end{equation}
where
\begin{align}
G(v)&:=\delta_{v \equiv 1 \ppmod{\lambda^3}} \cdot \Lambda(v) \widetilde{g}_4(v) \Big( \frac{\overline{v}}{|v|} \Big)^{\ell} \label{Gdef}; \\
H(w)&:=\delta_{w \equiv 1 \ppmod{\lambda^3}} \cdot \widetilde{g}_4(w) \Big(\frac{\overline{w}}{|w|} \Big)^{\ell} \sum_{\substack{bc=w \\ N(b) \leq u \\ b,c \equiv 1 \ppmod{\lambda^3}}} \mu(b). \label{Hdef}
\end{align}
Note that we suppressed the dependence on $\ell$ and $u$ in the notation in \eqref{Adef}--\eqref{Bdef}
and \eqref{Gdef}--\eqref{Hdef}.
Divisor bounds and \eqref{sqrootcancel}--\eqref{tildedef} imply that 
\begin{equation} \label{unnormalisedbd}
A(v), B(w), G(v), H(w) \ll X^{\varepsilon},
\end{equation}
uniformly in $\ell$ and $u$. All four functions  
are supported on squarefree elements in $\mathbb{Z}[i]$ that are congruent to $1 \ppmod{\lambda^3}$ by \eqref{sqrootcancel}.

We first estimate \eqref{2primeintermed}.
We dyadically partition $N(v) \sim V$ and $N(w) \sim W$, where
\begin{equation} \label{supportcond}
u/2 \leq V \leq u^2, \quad \text{and} \quad X/4 \leq VW \leq 2X. 
\end{equation}
Thus 
\begin{equation} \label{dyadicpart}
|\Sigma_{2^{\prime \prime},\beta}(X,\ell,u)| \leq \sum_{V,W \text{ dyadic}} \sum_{\substack{\eta,\gamma \in \{1,1+\lambda^3 \} \ppmod{4} \\ \eta \gamma \equiv \beta \ppmod{4}}} 
\Big |  \sum_{\substack{v,w \in \mathbb{Z}[i] \\  N(v) \sim V, \hspace{0.05cm} N(w) \sim W  \\ v \equiv \eta \ppmod{4} \\ w \equiv \gamma \ppmod{4}}}
A(v) B(w) \Big(\frac{w}{v} \Big)_2 R \Big( \frac{N(vw)}{X} \Big) \Big |.
\end{equation}
Define the Dirichlet polynomial,
\begin{equation} \label{Fdef}
P_{\eta,\gamma}(s;V,W):=
\sum_{\substack{v,w \in \mathbb{Z}[i] \\ N(v) \sim V, \hspace{0.1cm} N(w) \sim W  \\  v \equiv \eta \ppmod{4} \\ w \equiv \gamma \ppmod{4} }}
A(v) B(w) \Big( \frac{w}{v} \Big)_2 N(vw)^{-s}, \quad s \in \mathbb{C}.
\end{equation}
Mellin inversion of the smooth function $R$ gives,
\begin{equation*}
\sum_{\substack{v,w \in \mathbb{Z}[i] \\ N(v) \sim V, \hspace{0.05cm} N(w) \sim W \\ v \equiv \eta \ppmod{4} \\ w \equiv \gamma \ppmod{4} }}
A(v) B(w) \Big( \frac{w}{v} \Big)_2 R \Big( \frac{N(vw)}{X} \Big) \\
=\frac{1}{2 \pi i} \int_{-iX^{\varepsilon}}^{iX^{\varepsilon}} \widehat{R}(s) P_{\eta,\gamma}(s;V,W) X^s ds+O(X^{-1000}),
\end{equation*}
where the error term follows from the rapid decay of $\widehat{R}(s)$ in fixed vertical strips and \eqref{unnormalisedbd}. Thus 
\begin{equation} \label{perroncor}
 \sum_{\substack{v,w \in \mathbb{Z}[i] \\ N(v) \sim V, \hspace{0.05cm} N(w) \sim W \\ v \equiv \eta \ppmod{4} \\ w \equiv \gamma \ppmod{4} }} 
A(v) B(w) \Big( \frac{w}{v} \Big)_2 R \Big( \frac{N(vw)}{X} \Big)
 \ll X^{-1000}+ \int_{-iX^{\varepsilon}}^{iX^{\varepsilon}} |\widehat{R}(s)| |P_{\eta,\gamma}(s;V,W)| |ds|.
\end{equation}
Set
\begin{align} \label{tildeweight}
\widetilde{A}(v):=A(v) N(v)^{-it} \quad \text{and} \quad \widetilde{B}(w):=B(w) N(w)^{-it}.
\end{align}
From \eqref{unnormalisedbd} we find that
\begin{equation} \label{tildebd}
 \widetilde{A}(v), \widetilde{B}(w) \ll X^{\varepsilon},
\end{equation}
uniformly in $\ell, u$ and $t$. From \eqref{Fdef} we see that 
\begin{equation} \label{F1line}
P_{\eta,\gamma}(it;V,W)=
\sum_{\substack{N(v) \sim V \\ v \equiv \eta \ppmod{4} }} \widetilde{A}(v) \sum_{\substack{ N(w) \sim W \\ w \equiv \gamma \ppmod{4}}} \widetilde{B}(w) 
\Big(\frac{w}{v} \Big)_2.
\end{equation}
Applying the Cauchy-Schwarz inequality and using the fact that the weights $\widetilde{A}(v)$ and $\widetilde{B}(w)$
are supported on squarefree integers, the quadratic large sieve
(Theorem \ref{quadsieve}), and \eqref{tildebd}, we obtain
\begin{align} \label{sqbd}
& | P_{\eta,\gamma}(it;V,W)|^2 \nonumber \\
&= \Big | \sum_{\substack{ N(v) \sim V \\ v \equiv \eta \ppmod{4} }} \widetilde{A}(v) \sum_{\substack{ N(w) \sim W \\ w \equiv \gamma \ppmod{4}} } \widetilde{B}(w) 
\Big( \frac{w}{v} \Big)_2  \Big |^2 \nonumber \\
& \leq \Big( \sum_{\substack{ N(v) \sim V \\ v \equiv \eta \ppmod{4}}} |\widetilde{A}(v)|^2 \Big) \Big(  \sum_{\substack{ N(v) \sim V \\ v \equiv \eta \ppmod{4}}}
\mu^2(v) \Big | \sum_{\substack{ N(w) \sim W \\ w \equiv \gamma \ppmod{4} }} \widetilde{B}(w) \Big( \frac{w}{v} \Big)_2 \Big |^2  \Big) \nonumber \\
& \ll (VW)^{1+\varepsilon}(V+W),
\end{align}
uniformly in $t$.
We substitute the bound \eqref{sqbd} into \eqref{perroncor} 
and obtain,
\begin{equation} \label{dyadbd}
 \sum_{\substack{v,w \in \mathbb{Z}[i] \\ N(v) \sim V, \hspace{0.05cm} N(w) \sim W \\ v \equiv \eta \ppmod{4} \\ w \equiv \gamma \ppmod{4} }} 
A(v) B(w) \Big( \frac{w}{v} \Big)_2 R \Big( \frac{N(vw)}{X} \Big) \ll X^{\varepsilon} (VW)^{1/2} (V^{1/2}+W^{1/2}).
\end{equation}
Substituting \eqref{dyadbd} into \eqref{dyadicpart}, and using \eqref{supportcond}, we obtain the result
for $\Sigma_{2^{\prime \prime},\beta}(X,\ell,u)$. The argument for 
$\Sigma_{3,\beta}(X,\ell,u)$ is analogous.
\end{proof}

\section{Dirichlet series with level structure}
Recall the Dirichlet series $\psi^{(4)}_{\beta}(s,\nu,\ell)$ for $\beta \in \{1,1+\lambda^3 \} \ppmod{4}$
defined in \eqref{dir}.
For $\sigma:=\Re(s) >3/2$, $\alpha \in \mathbb{Z}[i]$ with $\alpha \equiv 1 \pmod{\lambda^3}$, and $0 \neq \nu \in \lambda^{-2} \mathbb{Z}[i]$,
consider
\begin{align} 
\psi^{(4)}_{\beta}(s,\nu,\ell;\alpha)&:=\sum_{\substack{c \in \mathbb{Z}[i] \\ c \equiv \beta \ppmod{4} \\ c \equiv 0 \ppmod{\alpha} }} \frac{g_4(\nu,c) ( \frac{\overline{c}}{|c|})^{\ell} }{N(c)^s};
\label{dirlevel} \\
\psi^{(4)}_{\beta \star}(s,\nu,\ell;\alpha)&:=\sum_{\substack{c \in \mathbb{Z}[i] \\ c \equiv \beta \ppmod{4} \\ (c,\alpha)=1 }}
\frac{g_4(\nu,c) ( \frac{\overline{c}}{|c|})^{\ell} }{N(c)^s}. \nonumber
\end{align}
If $\alpha=1$ we omit it (and the $\star$) from the notation.
\begin{lemma} \label{basiclem}
Suppose that $\ell \in \mathbb{Z}$ and $\beta \in \{1,1+\lambda^3\} \ppmod{4}$.
Let $\alpha \in \mathbb{Z}[i]$ and $\nu \in \lambda^{-2} \mathbb{Z}[i]$ be such that $\nu \neq 0$, $(\alpha,\nu)=1$, and $\alpha \equiv 1 \ppmod{\lambda^3}$ is squarefree. 
Let $C(\cdot,\cdot)$ be as given in \eqref{Cdef}, and
\begin{equation} \label{Deltadef}
\Delta(s,\ell;\alpha):=\prod_{\pi\mid \alpha} \Big(1 - N(\pi)^{3-4s}
\Big( \frac{\overline{\pi}}{|\pi|} \Big)^{4 \ell} \Big).
\end{equation}
Then for 
 $\sigma=\Re(s)>3/2$ we have the identities,
\begin{equation}
\psi^{(4)}_{\beta}(s,\nu,\ell;\alpha)= 
(-1)^{C(\alpha,\alpha \beta)} N(\alpha)^{-s} g_4(\nu,\alpha) \Big(\frac{\overline{\alpha}}{|\alpha|} \Big)^{\ell} 
\psi^{(4)}_{\alpha \beta \star}(s,\nu \alpha^2,\ell;\alpha); \label{id1}
\end{equation}
\begin{align}
& \psi^{(4)}_{\beta \star}(s,\alpha^2 \nu, \ell;\alpha) \Delta(s,\ell;\alpha) \nonumber \\
&=\sum_{d \mid \alpha} \mu(d)  (-1)^{C(d,d\beta)} N(d)^{2-3s} \overline{g_4 \Big( \frac{\nu \alpha^2}{d^2}, d \Big)} \Big( \frac{-1}{d} \Big)_4 \Big( \frac{\overline{d}}{|d|} \Big)^{3 \ell} 
\psi^{(4)}_{d \beta} \Big(s, \frac{\nu \alpha^2 }{d^2},\ell \Big); \label{id2} 
\end{align}
and
\begin{align}
& \psi^{(4)}_{\beta}(s,\nu,\ell;\alpha) \Delta(s,\ell;\alpha) \nonumber \\
&=N(\alpha)^{-s} (-1)^{C(\alpha,\alpha \beta)} \Big( \frac{\overline{\alpha}}{|\alpha|} \Big)^{\ell} \nonumber \\
& \times \sum_{d \mid \alpha} \mu(d) (-1)^{C(d,d \alpha \beta)+C(d,d \alpha)} N(d)^{3-3s} g_4 \Big(\nu,\frac{\alpha}{d} \Big) 
\Big(\frac{-1}{d} \Big)_4 \Big( \frac{\overline{d}}{|d|} \Big)^{3 \ell} \psi^{(4)}_{d \alpha \beta} \Big(s,\frac{ \nu \alpha^2 }{d^2},\ell \Big). \label{id3}
\end{align}
\end{lemma}
\begin{proof}
We first prove \eqref{id1}. If $c \equiv 1 \ppmod{\lambda^3}$ and 
$c \equiv 0 \ppmod{\alpha}$ with $(c,\nu)=1$,
then by \eqref{rel1} (with $r \rightarrow \nu$ and $\nu \rightarrow 1$), \eqref{rel2} and \eqref{rel4} we have $g_4(\nu, c)=0$ unless $c = \alpha c^{\prime}$ with $\mu^2(\alpha c^{\prime})=1$. 
Using this observation, \eqref{Cdef}, and \eqref{rel3} we obtain
\begin{align*}
 \psi^{(4)}_{\beta}(s,\nu,\ell;\alpha)&= \sum_{\substack{c \in \mathbb{Z}[i] \\ c \equiv \alpha \beta \ppmod{4} \\ (c, \alpha)=1}} 
 \frac{g_4(\nu, \alpha c) \big( \frac{\overline{\alpha c}}{|\alpha c|} \big)^{\ell}}{N(\alpha c)^s}\\
 &=(-1)^{C(\alpha,\alpha \beta)}   N(\alpha)^{-s} g_4(\nu, \alpha) \Big(\frac{\overline{\alpha}}{|\alpha|}  \Big)^{\ell}
 \sum_{\substack{c \in \mathbb{Z}[i]  \\ c \equiv \alpha \beta \ppmod{4}
 \\ (c, \alpha)=1}}  \frac{g_4(\nu \alpha^2, c) \big( \frac{\overline{c}}{|c|} \big)^{\ell}}{N(c)^s}  \\
 &=(-1)^{C(\alpha,\alpha \beta)}  N(\alpha)^{-s} g_4(\nu,\alpha) \Big(\frac{\overline{\alpha}}{|\alpha|}  \Big)^{\ell} 
 \psi^{(4)}_{\alpha \beta \star}(s,\nu \alpha^2,\ell;\alpha).
 \end{align*}
 
We now prove \eqref{id2}. Let $\gamma, \pi \in \mathbb{Z}[i]$ be such that $\gamma \equiv 1 \ppmod{\lambda^3}$, 
$\pi \equiv 1 \ppmod{\lambda^3}$ a prime, and $(\pi,\nu \gamma)=1$.
Then by \eqref{Cdef}--\eqref{rel4} we obtain,
\begin{align} \label{intermed1}
\psi^{(4)}_{\beta \star}(s,\nu,\ell;\gamma) &= \sum_{k = 0}^\infty \sum_{\substack{c \in \mathbb{Z}[i] \\ c \equiv \pi^k \beta \ppmod{4} \\ (c, \pi \gamma)=1}} \frac{g_4(\nu, \pi^k c) \big( \frac{\overline{\pi^k c}}{|\pi^k c|} \big)^{\ell}}{N(c \pi^k)^s} \nonumber \\
&=\sum_{\substack{c \in \mathbb{Z}[i] \\ c \equiv \beta \ppmod{4} \\ (c, \pi \gamma)=1}} \frac{g_4(\nu,c) \big(\frac{\overline{c}}{|c|} \big)^{\ell}}{N(c)^s}
+ \sum_{\substack{c \in \mathbb{Z}[i] \\
c \equiv \pi \beta \ppmod{4} \\ (c, \pi \gamma)=1}} \frac{g_4(\nu, \pi c) \big( \frac{\overline{\pi c}}{|\pi c|} \big)^{\ell}}{N(c \pi)^s} \nonumber \\
 &=\psi^{(4)}_{\beta \star}(s,\nu,\ell; \pi \gamma) \nonumber \\
 & + (-1)^{C(\pi,\pi \beta)} N(\pi)^{-s} g_4(\nu,\pi) \Big( \frac{\overline{\pi}}{|\pi|} \Big)^{\ell} \psi^{(4)}_{\pi \beta \star}(s,\pi^2 \nu,\ell;\pi \gamma). 
 \end{align} 
With a similar argument using \eqref{Cdef}--\eqref{rel4} we obtain,
\begin{align} \label{intermed2}
\psi^{(4)}_{\beta \star}(s,\pi^2 \nu,\ell;\gamma) &= \sum_{\substack{c \in \mathbb{Z}[i] \\ c \equiv \beta \ppmod{4}
\\ (c, \pi \gamma)=1}} \frac{g_4(\pi^2 \nu, c) \big( \frac{\overline{c}}{|c|} \big)^{\ell}}{N(c)^s}  + \sum_{\substack{c \in \mathbb{Z}[i] \\ c \equiv \pi \beta \ppmod{4} \\ (c, \pi \gamma)=1}} \frac{g_4(\pi^2 \nu, \pi^3 c) \big(\frac{
\overline{\pi^3 c}}{|\pi^3 c|} \big)^{\ell} }{N(c \pi^3)^s} \nonumber \\
& =\psi^{(4)}_{\beta \star}(s,\pi^2 \nu,\ell;\pi \gamma) \nonumber \\
& + (-1)^{C(\pi,\pi \beta)} N(\pi)^{-3s} g_4(\pi^2 \nu, \pi^3) \Big( \frac{\overline{\pi}}{|\pi|} \Big)^{3 \ell} 
\psi^{(4)}_{\pi \beta \star}(s,\nu,\ell;\pi \gamma),
 \end{align}
 since by \eqref{rel1} and \eqref{rel3} we have that
 \begin{equation*}
 g_4(\pi^2 \nu,\pi^3 c)=(-1)^{C(\pi^3,c)} g_4(\pi^2 \nu,\pi^3) g_4(\pi^8 \nu,c)=(-1)^{C(\pi,\pi \beta)} g_4(\pi^2 \nu,\pi^3) g_4(\nu,c).
 \end{equation*}
 After changing variable $\beta \ppmod{4} \rightarrow \pi \beta \ppmod{4}$ in the subscripts of \eqref{intermed1}, we use the result and \eqref{intermed2}
 to eliminate $\psi^{(4)}_{\pi \beta \star}(s,\nu,\ell;\pi \gamma)$. Using \eqref{rel1}, \eqref{rel4}, and \eqref{sqrootcancel}
 to simplify the resulting identity gives
\begin{align} \label{for-induction}
 & \psi^{(4)}_{\beta \star}(s,\pi^2 \nu,\ell;\pi \gamma) \Delta(s,\ell;\pi)
 \nonumber \\
 & = \psi^{(4)}_{\beta \star}(s,\pi^2 \nu,\ell;\gamma) 
 - (-1)^{C(\pi,\pi \beta)}  
 N(\pi)^{2-3s} \overline{g_4(\nu,\pi)} \Big(\frac{-1}{\pi} \Big)_4 \Big( \frac{\overline{\pi}}{|\pi|} \Big)^{3 \ell} \psi^{(4)}_{\pi \beta \star}(s,\nu,\ell;\gamma), \nonumber \\
& \qquad \qquad \qquad \qquad \qquad \qquad \qquad \qquad \qquad \qquad \qquad \qquad \qquad \text{for} \quad (\pi,\nu \gamma)=1.
 \end{align}
 
Recall that we need to prove \eqref{id2} holds for each $\ell \in \mathbb{Z}$, $\beta \in \{1,1+\lambda^3 \} \ppmod{4}$, and
all $\alpha \in \mathbb{Z}[i]$, $0 \neq \nu \in \lambda^{-2} \mathbb{Z}[i]$ such that
$\alpha \equiv 1 \ppmod{\lambda^3}$ is squarefree and $(\alpha,\nu)=1$.
We now induct on the number of primes in the factorisation of $\alpha$
in order to prove \eqref{id2}. Observe that \eqref{id2} is trivially true when $\alpha=1$. 
Setting $\gamma=1$ in \eqref{for-induction}
we obtain
\begin{align*}
 & \psi^{(4)}_{\beta \star}(s,\pi^2 \nu,\ell;\pi) \Delta(s,\ell;\pi) \nonumber \\
 & = \psi^{(4)}_{\beta}(s,\pi^2 \nu,\ell) 
 - (-1)^{C(\pi,\pi \beta) }  
 N(\pi)^{2-3s} \overline{g_4(\nu,\pi)} \Big(\frac{-1}{\pi} \Big)_4 \Big( \frac{\overline{\pi}}{|\pi|} \Big)^{3 \ell} \psi^{(4)}_{\pi \beta}(s,\nu,\ell),
 \end{align*}
 which establishes \eqref{id2} for any $\ell \in \mathbb{Z}$, $\beta \in \{1,1+\lambda^3 \} \ppmod{4}$,
 and $\pi \in \mathbb{Z}[i]$, $0 \neq \nu \in \lambda^{-2} \mathbb{Z}[i]$ such that
$\pi \equiv 1 \ppmod{\lambda^3}$ is prime and $(\pi,\nu)=1$.
Given an integer $k \geq 1$,
suppose that \eqref{id2} is true for all data $\ell,\beta,\nu,\alpha$ satisfying the given conditions 
and such that $\omega(\alpha) \leq k$. Let $\varpi \in \mathbb{Z}[i]$ be a prime with $\varpi \equiv 1 \ppmod{\lambda^{3}}$
 such that $(\varpi,\alpha \nu)=1$.
Using \eqref{for-induction} with $\nu \rightarrow \alpha^2 \nu$, $\gamma \rightarrow \alpha$, and $\pi \rightarrow \varpi$, we obtain 
\begin{align} \label{indstep}
& \Delta(s,\ell;\alpha) \Delta(s,\ell;\varpi) \psi^{(4)}_{\beta \star}(s, \alpha^2 \varpi^2 \nu,\ell;\alpha \varpi) \nonumber \\ 
&= \Delta(s,\ell;\alpha) \Big(\psi^{(4)}_{\beta \star} (s, \alpha^2 \varpi^2 \nu,\ell; \alpha) \nonumber \\
& - (-1)^{C(\varpi, \varpi \beta)} N(\varpi)^{2-3s} \overline{g_4(\alpha^2 \nu, \varpi)}  \Big(\frac{-1}{\varpi} \Big)_4 \Big( \frac{\overline{\varpi}}{|\varpi|} \Big)^{3 \ell} \psi^{(4)}_{\varpi \beta \star}(s,\alpha^2 \nu,\ell;\alpha) \Big).
\end{align}
Expanding the right side of \eqref{indstep}  
and then applying the inductive hypothesis twice, we obtain
\begin{align} \label{postexpansion}
& \sum_{d \mid \alpha} \mu(d) \Big( (-1)^{C(d,d\beta)} N(d)^{2-3s} \overline{g_4 \Big( \frac{\nu \varpi^2 \alpha^2}{d^2}, d \Big)} \Big( \frac{-1}{d} \Big)_4 \Big( \frac{\overline{d}}{|d|} \Big)^{3 \ell} 
\psi^{(4)}_{d \beta} \Big(s, \frac{\nu \varpi^2 \alpha^2 }{d^2},\ell \Big) \nonumber \\
&-(-1)^{C(\varpi, \varpi \beta)+C(d,\varpi d \beta)} N(\varpi d)^{2-3s}
 \overline{g_4(\alpha^2 \nu, \varpi)} \overline{g_4 \Big( \frac{\nu \alpha^2}{d^2}, d \Big)} \Big( \frac{-1}{\varpi d} \Big)_4 
 \Big( \frac{\overline{\varpi d}}{|\varpi d|} \Big)^{3 \ell} \psi^{(4)}_{\varpi d \beta} \Big(s, \frac{\nu \alpha^2 }{d^2},\ell \Big) \Big).
\end{align}
Using \eqref{rel3} we obtain
\begin{equation} \label{indgauss}
\overline{g_4(\alpha^2 \nu, \varpi)} \overline{g_4 \Big( \frac{\nu \alpha^2}{d^2}, d \Big)}=\overline{g_4 \Big(\frac{\alpha^2 \nu}{d^2},\varpi d \Big)} (-1)^{C(\varpi,d)}.
\end{equation}
After considering all possibilities $\varpi,d,\beta \in \{1,1+\lambda^3 \} \ppmod{4}$, we have the identity
\begin{equation} \label{Crel}
{C(\varpi,\varpi \beta)+C(d,\varpi d \beta)+C(\varpi,d)} \equiv C(\varpi d,\varpi d \beta) \ppmod{2}.
\end{equation}
Substituting \eqref{indgauss} and \eqref{Crel} into \eqref{postexpansion} we obtain 
\begin{equation*}
\sum_{d \mid \varpi \alpha} \mu(d) (-1)^{C(d,d\beta)} N(d)^{2-3s} \overline{g_4 \Big( \frac{\nu \varpi^2 \alpha^2}{d^2}, d \Big)} \Big( \frac{-1}{d} \Big)_4 \Big( \frac{\overline{d}}{|d|} \Big)^{3 \ell} 
\psi^{(4)}_{d \beta} \Big(s, \frac{\nu \varpi^2 \alpha^2 }{d^2},\ell \Big),
\end{equation*}
as required to prove \eqref{id2}.

We now prove \eqref{id3}. Combine \eqref{id1} and \eqref{id2}. We use 
the fact that $\alpha$ is squarefree, \eqref{bilaw}--\eqref{rel2}, 
and \eqref{sqrootcancel} to obtain the fact,
\begin{equation*}
g_4(\nu,\alpha) \overline{g_4 \Big(\frac{\nu \alpha^2}{d^2},d \Big)}=(-1)^{C(d,d \alpha)} N(d) g_4 \Big(\nu,\frac{\alpha}{d} \Big).
\end{equation*}
The identity \eqref{id3} now follows. 
\end{proof}

\begin{prop} \label{levelprop}
Suppose that $\ell \in \mathbb{Z}$ and $\beta \in \{1,1+\lambda^3 \} \ppmod{4}$. Let $\alpha \in \mathbb{Z}[i]$ and $\nu \in \lambda^{-2} \mathbb{Z}[i]$ be such that 
$\nu \neq 0$, $(\alpha,\nu)=1$, and $\alpha \equiv 1 \ppmod{\lambda^3}$ is squarefree. 
The functions $\psi^{(4)}_{\beta}(s,\nu,\ell;\alpha)$ in \eqref{dirlevel} can be meromorphically continued to $\mathbb{C}$:
if $\ell \neq 0$ they are holomorphic in the half-plane $\sigma>1$, and if $\ell=0$ they 
have at most a (simple) pole at $s=5/4$ in the half-plane $\sigma>1$. 

The residue
$p_{\beta}(\nu;\alpha):=\emph{Res}_{s=5/4} \psi^{(4)}_{\beta}(s,\nu,0;\alpha)$ is 
given by
\begin{align} \label{residueprop}
p_{\beta}(\nu;\alpha)&=(-1)^{C(\alpha,\alpha \beta)} \Delta(5/4,0;\alpha)^{-1} N(\alpha)^{-1}  \nonumber \\
& \times \sum_{d \mid \alpha} \mu(d) (-1)^{C(d,d \alpha \beta)+C(d,d \alpha)} N(d)^{-1}
\Big(\frac{-1}{d} \Big)_4 \nonumber \\
& \times \begin{cases}
 \psi^{(4)}_{d \alpha \beta}(\nu)  & \emph{if} \quad \alpha/d \equiv 1 \ppmod{4} \\
 (-1)^{\tfrac{N(d \alpha \beta)-1}{4}} \psi^{(4)}_{(1+\lambda^3)d \alpha \beta}(\nu) &
 \emph{if} \quad \alpha/d \equiv 1+\lambda^3 \ppmod{4},
 \end{cases}
 \end{align}
where $C(\cdot,\cdot)$ is given in \eqref{Cdef}, $\psi_{\beta}^{(4)}(\nu)$ in \eqref{res1},
and $\Delta(s,\ell;\alpha)$ in
\eqref{Deltadef}. 
 
For any $\varepsilon>0$ we have,
\begin{align} \label{convexbdlevel}
\psi^{(4)}_{\beta}(s,\nu,\ell;\alpha) & \ll_{\varepsilon, \emph{ord}_{\lambda}(\nu)}
N(\nu)^{(1/2)(3/2-\sigma)+\varepsilon} N(\alpha)^{2-2 \sigma+\varepsilon}  (|s|^2+\ell^2+1)^{(3/2)(3/2-\sigma)+\varepsilon} \nonumber \\
 & \quad \text{for} \quad 1+\varepsilon<\sigma<3/2+\varepsilon/2 \quad \text{and} \quad  |s- 5/4 |>1/8. 
\end{align}
\end{prop}

\begin{remark}
Our Proposition \ref{levelprop} is the quartic analogue of the cubic \cite[Lemma 4]{HBP}. 
After the change of variables $s \rightarrow s+1/2$ and a subsequent rescaling by $N(\alpha)^s$ (see \eqref{littlefdef}), 
the convexity bound for the function in \eqref{littlefdef} is $\operatorname{GL}_1$ in the $\nu$-aspect,
$\operatorname{GL}_2$ in the $\alpha$-aspect, and $\operatorname{GL}_6$ in the $(|s|+|\ell|+1)$-aspect.
In the cubic case, the corresponding convexity bound is
$\operatorname{GL}_1$ in both the $\nu$ and $\alpha$-aspects, 
and $\operatorname{GL}_4$ in the $(|s|+|\ell|+1)$-aspect.
\end{remark}

\begin{proof}
The claims regarding the meromorphic continuation to all of $\mathbb{C}$ and the
location of the poles in the half-plane $\sigma>1$ follow from 
\eqref{id3}, \eqref{connect1}--\eqref{connect2}, \eqref{Zi1defn}--\eqref{Ginfty},
Proposition \ref{convexity}, and the fact that
$(\Delta(s,\ell;\alpha) G_{\infty}(s,\ell) \zeta_{\mathbb{Q}(i),\lambda}(4s-3,\ell))^{-1}$
is holomorphic in the half plane $\sigma>1$.

The equality \eqref{residueprop} follows from 
combining \eqref{id3}, \eqref{res1},
Suzuki's property \eqref{sqrel}, \eqref{rel1}
and \eqref{sqrootcancel}.
The summand corresponding to $d \mid \alpha$ such that $\alpha/d \equiv 1 \ppmod{4}$
is
\begin{align*}
& (-1)^{C(\alpha,\alpha \beta)} N(\alpha)^{-5/4} \Delta(5/4,0;\alpha)^{-1}  \\
& \times \mu(d) (-1)^{C(d,d \alpha \beta)+C(d,d \alpha)} N(d)^{-3/4} g_4 \Big(\nu,\frac{\alpha}{d} \Big) 
\Big(\frac{-1}{d} \Big)_4  \psi^{(4)}_{d \alpha \beta} \Big( \frac{\nu \alpha^2}{d^2} \Big) \\
&=(-1)^{C(\alpha,\alpha \beta)} N(\alpha)^{-5/4} \Delta(5/4,0;\alpha)^{-1} \\
& \times \mu(d) (-1)^{C(d,d \alpha \beta)+C(d,d \alpha)} N(d)^{-3/4} N \Big( \frac{\alpha}{d} \Big)^{1/4}
\Big(\frac{-1}{d} \Big)_4  \psi^{(4)}_{d \alpha \beta} \left( \nu \right) \\
&=(-1)^{C(\alpha,\alpha \beta)} N(\alpha)^{-1}  \Delta(5/4,0;\alpha)^{-1} \\
& \times \mu(d) (-1)^{C(d,d \alpha \beta)+C(d,d \alpha)} N(d)^{-1} 
\Big(\frac{-1}{d} \Big)_4  \psi^{(4)}_{d \alpha \beta} \left( \nu \right).
\end{align*}
We have a similar computation for the case $\alpha/d \equiv 1 + \lambda^3 \ppmod 4$.

Observe that $\Delta(s,\ell;\alpha)^{-1} \ll 1$ for $\sigma>1+\varepsilon$. Combining \eqref{rel1},
\eqref{sqrootcancel}, \eqref{connect1}--\eqref{connect2}, \eqref{convexbd}, and \eqref{id3} we obtain
\begin{align}
\psi^{(4)}_{\beta}(s,\nu,\ell;\alpha) &\ll_{\text{ord}_{\lambda}(\nu)} 
 (|s|^2+\ell^2+1)^{(3/2)(3/2-\sigma)+\varepsilon} N(\alpha)^{-\sigma} \nonumber \\
& \times \sum_{d \mid \alpha} N(d)^{3-3 \sigma} N \Big( \frac{\alpha}{d} \Big)^{1/2}
N \Big(\frac{\nu \alpha^2}{d^2} \Big)^{(1/2)(3/2-\sigma)+\varepsilon} \nonumber \\
& \ll N(\nu)^{(1/2)(3/2-\sigma)+\varepsilon} N(\alpha)^{2-2\sigma+\varepsilon}  (|s|^2+\ell^2+1)^{(3/2)(3/2-\sigma)+\varepsilon},
\end{align}
for $1+\varepsilon<\sigma<3/2+\varepsilon/2$ and $|s- 5/4 |>1/8$, as required.
\end{proof}

\begin{corollary} \label{resbd}
Let $\beta \in \{1,1+\lambda^3\} \ppmod{4}$, $0 \neq \nu \in \lambda^{-2} \mathbb{Z}[i]$, and $\varepsilon>0$.
Then for all $\alpha \in \mathbb{Z}[i]$ squarefree with $\alpha \equiv 1 \ppmod{\lambda^3}$
and $(\alpha,\nu)=1$ we have 
\begin{equation*} 
p_{\beta}(\nu;\alpha) \ll_{\varepsilon,\nu} N(\alpha)^{-1+\varepsilon}.
\end{equation*}
\end{corollary}
\begin{remark}
One obtains the inferior bound $p(1;\alpha) \ll N(\alpha)^{-1/2+\varepsilon}$ if only the 
convexity bound from Proposition \ref{levelprop} is used  (see the bottom display
on \cite[pg.~200]{Pat3}).
\end{remark}

\section{Voronoi formula and Average Type-I estimates} \label{avgtype1sec}
Let $R:(0,\infty) \rightarrow \mathbb{C}$ be a smooth function with compact support
in $[1,2]$, $\ell \in \mathbb{Z}$, $\beta \in \{1,1+\lambda^3 \} \ppmod{4}$, 
and $\alpha \in \mathbb{Z}[i]$ squarefree with $\alpha \equiv 1 \ppmod{\lambda^3}$.
Recall the pointwise (in $\alpha$) Type-I sum $F_{\beta}(x,\ell;\alpha)$ in \eqref{Felldefn}, its corresponding 
Dirichlet series $\psi^{(4)}_{\beta}(s,1,\ell;\alpha)$ in \eqref{dirlevel}, and
the Dirichlet polynomial $\Delta(s,\ell;\alpha)$ in \eqref{Deltadef}. 

In this section we prove estimates for the terms $\Sigma_{1,\beta}$
and $\Sigma_{2^{\prime},\beta}$ appearing in \eqref{Vauid}. 
We exploit the averaging over $\alpha$ present in \eqref{sigma1bd} and \eqref{sigma2primebd}
 for additional savings using the quadratic large sieve.
Set
\begin{align}
f_{\beta}(s,\ell;\alpha)&:=N(\alpha)^s \psi^{(4)}_{\beta}(s+1/2,1,\ell;\alpha); \label{littlefdef}  \\
h_{\beta}(s,\ell;\alpha)&:=\zeta_{\mathbb{Q}(i),\lambda}(4s-1,\ell) \Delta(s+1/2,\ell;\alpha)
f_{\beta}(s,\ell;\alpha), \label{littlegdef} 
\end{align}
where $\zeta_{\mathbb{Q}(i),\lambda}(s,\ell)$ is defined by \eqref{def-zeta-lambda}. We now have a preparatory Lemma. 

\begin{lemma} \label{type1init2}
For $x \geq 1$, $1 \leq u \leq x$, and $0<\varepsilon<1/100$, we have 
\begin{equation} \label{residuethmstate}
\sum_{\substack{ N(\alpha) \leq u \\ \alpha \equiv 1 \ppmod{\lambda^3}}} \mu^2(\alpha) |F_{\beta}(x,\ell;\alpha)|  \ll_{\varepsilon,R}
\delta_{\ell=0} \cdot  x^{3/4+\varepsilon}+I(x,\ell,u)+(x (|\ell|+1))^{-100},
\end{equation}
where
\begin{align*}
I(x,\ell,u) 
&:= x^{1/2+\varepsilon} \int_{1/2+\varepsilon-i(x (|\ell|+1))^{\varepsilon}}^{1/2+\varepsilon+i(x (|\ell|+1))^{\varepsilon}}
|\widehat{R}(s)| \Big(
\sum_{\substack{N(\alpha) \leq u \\ \alpha \equiv 1 \ppmod{\lambda^3} }}
\mu^2(\alpha) |h_{\beta}(s,\ell;\alpha)|^2 \Big)^{1/2} | ds |. 
\end{align*}
\end{lemma}
\begin{proof}
Let $\alpha \in \mathbb{Z}[i]$ be squarefree and satisfy $N(\alpha) \leq u$ and $\alpha \equiv 1 \ppmod{\lambda^3}$.
Mellin inversion of the smooth function $R$ gives,
\begin{equation} \label{perronapply}
F_{\beta}(x,\ell;\alpha)=\frac{1}{2 \pi i} \int_{1+\varepsilon-i(x (|\ell|+1))^{\varepsilon}}^{1+\varepsilon+i (x (|\ell|+1))^{\varepsilon}} f_{\beta}(s,\ell;\alpha) \widehat{R}(s)
\Big( \frac{x}{N(\alpha)} \Big )^s ds+O((x (|\ell|+1))^{-1000}),
\end{equation}
where the error term follows from the rapid decay of $\widehat{R}(s)$ in fixed vertical strips,  \eqref{littlefdef},
and the absolute-convergence bound \eqref{absconvbd}, together with \eqref{id3}.
We shift the contour to $\sigma=1/2+\varepsilon$, passing a possible simple pole of the integrand of \eqref{perronapply} at $s=3/4$
when $\ell=0$ by \eqref{littlefdef} and Proposition \ref{levelprop}, with residue  
\begin{equation*}
\delta_{\ell=0} \cdot p(1;\alpha) \widehat{R} \Big( \frac{3}{4} \Big) x^{3/4},
\end{equation*}
where $p(1;\alpha)$ is computed in
\eqref{residueprop}.
We obtain
\begin{align} \label{residuethm}
F_{\beta}(x,\ell;\alpha)&=\delta_{\ell=0} \cdot p(1;\alpha) \widehat{R} \Big( \frac{3}{4} \Big) x^{3/4}
+ \frac{1}{2 \pi i} \int_{1/2+\varepsilon-i(x (|\ell|+1))^{\varepsilon}}^{1/2+\varepsilon+i(x (|\ell|+1))^{\varepsilon}}
f_{\beta}(s,\ell;\alpha) \widehat{R}(s) \Big( \frac{x}{N(\alpha)} \Big )^s ds \nonumber \\
&+O((x (|\ell|+1))^{-1000}).
\end{align}

We take the absolute value of both sides of \eqref{residuethm}, and apply the triangle inequality to the right side, and then
introduce the finite sum over squarefree $\alpha \in \mathbb{Z}[i]$ with $\alpha \equiv 1 \ppmod{\lambda^3}$ and $N(\alpha) \leq u$ on both sides.
We interchange the integration over $s$ and sum over $\alpha$ by absolute convergence, and then 
apply the Cauchy-Schwarz inequality to obtain
\begin{equation} \label{fhineq}
 \sum_{\substack{N(\alpha) \leq u \\ \alpha \equiv 1 \ppmod{\lambda^3} }} 
N(\alpha)^{-1/2-\varepsilon} \mu^2(\alpha) |f_{\beta}(s,\ell;\alpha)|
\ll \Big(\sum_{\substack{N(\alpha) \leq u \\ \alpha \equiv 1 \ppmod{\lambda^3} }}
\mu^2(\alpha) |h_{\beta}(s,\ell;\alpha)|^2 \Big)^{1/2}.
\end{equation}
To obtain the last display we used the fact that
$|\Delta(s+1/2,\ell;\alpha) \zeta_{\mathbb{Q}(i),\lambda}(4s-1,\ell)| \asymp 1 $
 for $\sigma=1/2+\varepsilon$. After having used the triangle inequality in \eqref{residuethm}, we then apply \eqref{fhineq}. The
second and third terms on the right side of 
\eqref{residuethmstate} follow. Using Corollary \ref{resbd} we have
\begin{equation*}
\sum_{\substack{N(\alpha) \leq u \\ \alpha \equiv 1 \ppmod{\lambda^3}}} 
\delta_{\ell=0} \cdot \mu^2(\alpha) \Big | \widehat{R} \Big( \frac{3}{4} \Big) \Big |  |p(1;\alpha)| x^{3/4} 
\ll \delta_{\ell=0} \cdot  u^{\varepsilon} x^{3/4},
\end{equation*}
which gives the first term on the right side of \eqref{residuethmstate}.
The result now follows.
\end{proof}

We now prove a Voronoi formula for $h_{\beta}(s,\ell;\alpha)$.
Write the Dirichlet polynomial in \eqref{Deltadef} as,
\begin{equation} \label{deltaexpand}
\Delta(s,\ell;\alpha)=\sum_{m \mid \alpha} q(m,\ell) N(m)^{3-4s},
\end{equation}
where 
\begin{equation*}
q(m, \ell):= \mu(m) \Big( \frac{\overline{m}}{|m|} \Big)^{4 \ell}.
\end{equation*}

\begin{lemma} \label{hvoronoi}
For $\ell \in \mathbb{Z}$, $\beta \in \{1,1+\lambda^3 \} \ppmod{4}$, 
and $\alpha \in \mathbb{Z}[i]$ squarefree with $\alpha \equiv 1 \ppmod{\lambda^3}$,
let $h_{\beta}(s,\ell;\alpha)$ be as in  
\eqref{littlegdef}.
Suppose $0<\varepsilon<1/100$ and 
$s=\sigma+it$ with $\sigma=1/2+\varepsilon$ and $t\in\mathbb R$.
Then for $x \geq 1$ we have,
\vspace{-0.125cm}
\begin{align}
& (-1)^{C(\alpha,\alpha \beta)} \widetilde{g}_4(\alpha) \Big( \frac{\overline{\alpha}}{|\alpha|} \Big)^{\ell} \sum_{m \mid \alpha}
q(m,\ell) N(m)^{1-4s}  \sum_{\substack{d \in \mathbb{Z}[i] \\ d \equiv 1 \ppmod{\lambda^3}}} \Big( \frac{\overline{d}}{|d|} \Big)^{4 \ell} N(d)^{1-4s}  \nonumber \\
& \times \sum_{\substack{c \in \mathbb{Z}[i] \\ c \equiv \alpha \beta \ppmod{4}}} 
N(c)^{-s} \widetilde{g}_4(c) \Big( \frac{\overline{c}}{|c|} \Big)^{\ell} 
\Big( \frac{\alpha}{c} \Big)_2 \cdot e^{-N(cd^4 m^4)/x}  \nonumber \\
&=h_{\beta}(s,\ell;\alpha)+O_{\varepsilon} ( \delta_{\ell=0} \cdot e^{-|t|} \cdot N(\alpha)^{-1/4+\varepsilon} x^{3/4-\sigma}) +
N(\alpha)^{1-2s} (-1)^{C(\alpha,\alpha \beta)} \Big( \frac{\overline{\alpha}}{|\alpha|} \Big)^{\ell}  \nonumber \\
& \times \sum_{m \mid \alpha}
N(m)^{-s} \mu(m) (-1)^{C(m,m \alpha \beta)+C(m,m \alpha)} \widetilde{g}_4 \Big(\frac{\alpha}{m} \Big) 
\Big(\frac{-1}{m} \Big)_4 \Big( \frac{\overline{m}}{|m|} \Big)^{3 \ell} \Big( \frac{\overline{\alpha/m}}{|\alpha/m|} \Big)^{2 \ell} \nonumber \\
& \times \sum_{\substack{ j=1,\ldots,24 \\ \eta \in \{1,1+\lambda^3\} \ppmod{4} }} \sum_{(a,b) \in S(\kappa_j,1)}
\mathscr{H}(3/2-s,\eta,\ell,\kappa_j,(a,b);\alpha,m)
\sum_{\substack{d \in \mathbb{Z}[i] \\ d \equiv 1 \ppmod{\lambda^3} } } N(d)^{4s-3} \Big( \frac{\overline{d}}{|d|} \Big)^{-4 \ell}  \nonumber \\
& \times \sum_{\substack{c \in \mathbb{Z}[i] \\ c \equiv \eta \ppmod{4}}} N(c)^{s-1} \widetilde{g}_4 \Big ( \frac{ (-1)^a \lambda^{2b} \alpha^2}{m^2},c \Big) 
\Big(\frac{\overline{c}}{|c|} \Big)^{-\ell}
\cdot \mathscr{T}_{m \alpha \beta} \Big(s,\frac{N(\alpha^2 m)}{2^b x N(cd^4)},\ell;j \Big),
\end{align}
\vspace{-0.375cm}
where for $y>0$,
\begin{align} \label{Tbetadef}
\mathscr{T}_{\beta}(s,y,\ell;j)
&:= \frac{1}{2 \pi i} \int_{-\sigma-2 \varepsilon-i \infty}^{-\sigma-2 \varepsilon+i \infty}  
y^{-w}
 \frac{G_{\infty}(3/2-s-w,-\ell)}{G_{\infty}(1/2+s+w,\ell)} \Gamma(w) \nonumber \\
& \times
\Big(\begin{cases}
(-1)^{\ell} A_{ji_1}(2^{-s-w-1/2},\ell) & \emph{if} \quad \beta \equiv 1 \ppmod{4} \\
(-1)^{\ell+1} A_{ji_2}(2^{-s-w-1/2},\ell) & \emph{if} \quad \beta \equiv 1+\lambda^3 \ppmod{4}
\end{cases} \Big ) dw,
\end{align}
the functions $G_{\infty}(z,\ell)$ are given in \eqref{Ginfty},
$1 \leq i_1,i_2 \leq 24$ are such that $\kappa_{i_1}=0$ and $\kappa_{i_2}=\tfrac{\lambda^4}{1+\lambda^3}$,
the functions $A_{ij}(2^{-z},\ell)$ are those appearing in \eqref{funceqn},
\begin{align} \label{Hdefn}
&\mathscr{H}(s,\eta,\ell,\kappa_j,(a,b);\alpha,m) \nonumber \\
&=
\overline{\Big(\frac{\alpha_j}{\gamma_j} \Big)}_4 \check{e} \Big({-\frac{\gamma_j^{\prime} \alpha_j (\alpha^2/m^2) }{\lambda^4} } \Big)  \nonumber \\
& \times
\begin{cases}
1 & \emph{if } \gamma_j=-1 \emph{ and } \eta \equiv 1 \ppmod{4} \\
1 & \emph{if } \gamma_j=-1-\lambda^3, -1+\lambda^3 \emph{ and } \eta \equiv 1+\lambda^3 \ppmod{4}  \\
(-1)^{\ell}  &
\emph{if } \gamma_j=1 \emph{ and } \eta \equiv 1 \ppmod{4} \\ 
(-1)^{\ell+1}
& \emph{if } \gamma_j=1+\lambda^3,1-\lambda^3 \emph{ and } \eta \equiv 1+\lambda^3 \ppmod{4}  \\
0 & \emph{otherwise} 
\end{cases} \nonumber \\
&+\overline{\Big(\frac{-\gamma_j}{\alpha_j} \Big)}_4 i^{\ell a} \Big( \frac{\overline{\lambda}^b}{|\lambda^b|} \Big)^{-\ell}
2^{-bs} \Big( \frac{i^a \lambda^b}{\alpha/m} \Big)_2 \nonumber \\
& \times
\begin{cases}
\Gamma_{1}(1,i^a \lambda^b) &  \emph{if } (\gamma_j,\lambda) \neq 1, \alpha_j=1, \emph{ and } \eta \equiv 1 \ppmod{4}  \\
\Gamma_{1+\lambda^3}(1,i^a \lambda^b) & \emph{if } (\gamma_j,\lambda) \neq 1, \alpha_j=1+\lambda^3, \emph{ and } \eta \equiv 1 \ppmod{4}  \\
\Gamma_{1+\lambda^3}(1,i^a \lambda^b) & \emph{if } (\gamma_j,\lambda) \neq 1, \alpha_j=1, \emph{ and } \eta \equiv 1+\lambda^3 \ppmod{4}  \\
-\Gamma_{1}(1,i^a \lambda^b) & \emph{if } (\gamma_j,\lambda) \neq 1, \alpha_j=1+\lambda^3, \emph{ and }
\eta \equiv 1+\lambda^3 \ppmod{4}
\end{cases},
\end{align}
only the first summand in \eqref{Hdefn} is included when $(\gamma_j,\lambda)=1$,
and only the second when $(\gamma_j,\lambda)\neq1$.
In the first case, $\gamma_j^{\prime}$ denotes the inverse of $\gamma_j\ppmod4$.
The $\Gamma_{\beta}(\nu,i^a \lambda^b)$ are as in \eqref{gammagauss}, and the $S(\kappa_j,\nu)$ appear in
\eqref{suppdef}.
\end{lemma}
\begin{remark}
The squarefree $\alpha\in\mathbb Z[i]$ should not be confused with $\alpha_j$,
the numerator of the cusp $\kappa_j=\alpha_j/\gamma_j$ in \eqref{cusps}.
\end{remark}
\begin{proof}

Observe that for $\Re(s+w)>1$ we have,
\begin{align} \label{hbetaexpand}
h_{\beta}(s+w,\ell;\alpha) &= N(\alpha)^{s+w} \sum_{\substack{m \in \Z[i]\\ m \mid \alpha}} q(m, \ell) N(m)^{1-4(s+w)} 
\sum_{\substack{d \in \Z[i] \\ d \equiv 1 \ppmod{\lambda^3}}} \Big( \frac{\overline{d}}{|d|} \Big)^{4\ell} N(d)^{1-4(s+w)} \nonumber \\
& \times \sum_{\substack{c \in \Z[i]\\ c \equiv \alpha \beta \ppmod 4}}  {g}_4( \alpha c) \Big( \frac{\overline{ \alpha c}}{|\alpha c|} \Big)^{\ell} N(\alpha c)^{-s-w-1/2} \nonumber \\
& =   (-1)^{C(\alpha, \alpha \beta)} \widetilde{g}_4(\alpha) \Big( \frac{\overline{ \alpha }}{|\alpha |} \Big)^\ell  \sum_{\substack{m \in \Z[i]\\ m \mid \alpha}} q(m, \ell) N(m)^{1-4s} 
\sum_{\substack{d \in \Z[i] \\ d \equiv 1 \ppmod{\lambda^3}}} \Big( \frac{\overline{d}}{|d|} \Big)^{4\ell} N(d)^{1-4s} \nonumber \\
& \times \sum_{\substack{c \in \Z[i] \\c \equiv \alpha \beta \ppmod 4}}  N(c)^{-s} \widetilde{g}_4( c) \Big( \frac{\overline{ c}}{| c|} \Big)^\ell 
\Big( \frac{\alpha}{c} \Big)_2 N(cd^4 m^4)^{-w},
\end{align}
where we used \eqref{bilaw}--\eqref{Cdef} and \eqref{rel2} to write
\begin{equation*}
g_4( \alpha c)= (-1)^{C(\alpha, c)} \Big( \frac{\alpha}{c} \Big)_2 g_4(\alpha) g_4(c),
\end{equation*}
as well as \eqref{deltaexpand} in the computation. Using Mellin inversion and \eqref{hbetaexpand} we obtain
\begin{align} \label{mellin}
&  (-1)^{C(\alpha, \alpha \beta)} \widetilde{g}_4(\alpha) \Big( \frac{\overline{\alpha}}{|\alpha|} \Big)^{\ell} \sum_{m \mid \alpha}
q(m,\ell) N(m)^{1-4s} \sum_{\substack{ d \in \mathbb{Z}[i] \\ d \equiv 1 \ppmod{\lambda^3} }} 
\Big( \frac{\overline{d}}{|d|} \Big)^{4 \ell} N(d)^{1-4s} \nonumber \\
& \times \sum_{\substack{c \in \mathbb{Z}[i] \\ c \equiv \alpha \beta \ppmod{4} }} 
 N(c)^{-s} \widetilde{g}_4(c) \Big( \frac{\overline{c}}{|c|} \Big)^{\ell} \Big( \frac{\alpha}{c} \Big)_2
 e^{-N(cd^4 m^4)/x} \nonumber \\
&=\frac{1}{2 \pi i} \int_{1-i \infty}^{1+i \infty} h_{\beta}(s+w,\ell;\alpha) \Gamma(w) x^{w} dw.
\end{align}

Using \eqref{littlefdef}--\eqref{littlegdef}, \eqref{id3}, and \eqref{tildedef}, we see that
\begin{align*}
& h_{\beta}(s+w,\ell;\alpha) \nonumber \\
&=N(\alpha)^{s+w} \zeta_{\mathbb{Q}(i),\lambda}(4(s+w)-1,\ell) \Delta(s+w+1/2,\ell;\alpha) \psi^{(4)}_{\beta}(s+w+1/2,1,\ell;\alpha) \nonumber \\
&=(-1)^{C(\alpha,\alpha \beta)} \Big( \frac{\overline{\alpha}}{|\alpha|} \Big)^{\ell} \sum_{m \mid \alpha} \mu(m) (-1)^{C(m,m\alpha \beta)+C(m,m \alpha)} N(m)^{1-3(s+w)} \widetilde{g}_4 \Big(\frac{\alpha}{m} \Big)
\Big( \frac{-1}{m} \Big)_4 \Big(\frac{\overline{m}}{|m|}  \Big)^{3 \ell} \nonumber \\
& \times \zeta_{\mathbb{Q}(i),\lambda}(4(s+w)-1,\ell) \psi^{(4)}_{m \alpha \beta}(s+w+1/2,\alpha^2/m^2,\ell),
\end{align*}
and then using \eqref{connect1}--\eqref{connect2} and \eqref{Zi1defn} we obtain the expression,
\begin{align} \label{funcprep} 
& h_{\beta}(s+w,\ell;\alpha) \nonumber \\
&=(-1)^{C(\alpha,\alpha \beta)} \Big( \frac{\overline{\alpha}}{|\alpha|} \Big)^{\ell} \sum_{m \mid \alpha} \mu(m) (-1)^{C(m,m \alpha \beta)+C(m,m \alpha)}
N(m)^{1-3(s+w)} \widetilde{g}_4 \Big(\frac{\alpha}{m} \Big) \Big(\frac{-1}{m} \Big)_4 \Big( \frac{\overline{m}}{|m|} \Big)^{3 \ell} \nonumber \\ 
& \times
\begin{cases} 
(-1)^{\ell} \mathcal{V} Z_{i_1 1}(s+w+1/2, \alpha^2/m^2,\ell) & \text{if} \quad m \alpha \beta \equiv 1 \ppmod{4}  \\
(-1)^{\ell+1} \mathcal{V}  Z_{i_2 1}(s+w+1/2, \alpha^2/m^2,\ell) & \text{if} \quad m \alpha \beta \equiv 1+\lambda^3 \ppmod{4}
\end{cases}.
\end{align}
Thus from \eqref{funcprep} and Remark \ref{Zremark}, we see the only possible pole of  $h_{\beta}(s+w,\ell;\alpha)$ 
occurs when $\ell=0$ and is at $w=3/4-s \neq 0$.
We now move the contour in \eqref{mellin} to $\Re(w)=-\sigma-2\varepsilon$.
Since $-1<-\sigma-2\varepsilon<0$, no negative-integer pole of $\Gamma(w)$ is crossed.
We pass a simple pole at $w=0$
with residue 
\begin{equation} \label{zeropole}
h_{\beta}(s,\ell;\alpha),
\end{equation}
and also the possible simple pole at
$w=3/4-s$ when $\ell=0$. The residue of the integrand in \eqref{mellin} at $w=3/4-s$ is
\begin{align} \label{34res}
& \delta_{\ell=0} \cdot \zeta_{\mathbb{Q}(i),\lambda}(2,0) N(\alpha)^{3/4} \Delta(5/4,0;\alpha) p_{\beta}(1;\alpha) \Gamma(3/4-s) x^{3/4-s} \nonumber \\
& \ll \delta_{\ell=0} \cdot N(\alpha)^{-1/4+\varepsilon} x^{3/4-\sigma} e^{-|t|},
\end{align}
where $p(1;\alpha)$ is as in \eqref{residueprop},
and the estimate in \eqref{34res} follows from
Corollary \ref{resbd} and a crude application of Stirling's formula \cite[(5.11.1)]{NIST:DLMF}.

After using \eqref{hatdef}--\eqref{funceqn} on the right side of \eqref{funcprep}
we obtain the functional equation
\begin{align} \label{gfunceqn}
 & h_{\beta}(s+w,\ell;\alpha) \nonumber \\
 &=N(\alpha)^{1-2(s+w)} (-1)^{C(\alpha,\alpha \beta)} \Big( \frac{\overline{\alpha}}{|\alpha|} \Big)^{\ell} \frac{G_{\infty}(3/2-s-w,-\ell)}{G_{\infty}(s+w+1/2,\ell)} \nonumber \\
 & \times \sum_{m \mid \alpha}  
 N(m)^{-(s+w)} \mu(m) (-1)^{C(m,m \alpha \beta)+C(m,m \alpha)} \widetilde{g}_4 \Big(\frac{\alpha}{m} \Big) 
\Big(\frac{-1}{m} \Big)_4 \Big( \frac{\overline{m}}{|m|} \Big)^{3 \ell}  \Big( \frac{\overline{\alpha/m}}{|\alpha/m|} \Big)^{2 \ell} \nonumber \\
& \times \sum_{j=1}^{24} 
 \mathcal{V} Z_{j1}(3/2-s-w, \alpha^2/m^2,-\ell) \nonumber \\ 
& \times
\begin{cases}
 (-1)^{\ell} A_{ji_1}(2^{-s-w-1/2},\ell)  & \text{if} \quad m \alpha \beta \equiv 1 \ppmod{4} \\
 (-1)^{\ell+1} A_{ji_2}(2^{-s-w-1/2},\ell) & \text{if} \quad m \alpha \beta \equiv 1+\lambda^3 \ppmod{4}
\end{cases}.
\end{align}

We use Remark \ref{lincombo} (see \eqref{coprimecusp}--\eqref{gammatransform})
and \eqref{Zi1defn}, to express each series 
\begin{equation*}
\mathcal{V} Z_{j1}(3/2-s-w, \alpha^2/m^2,-\ell),
\end{equation*}
as a finite $\mathbb{C}$-linear combination of the functions
\begin{equation} \label{series1}
\zeta_{\mathbb{Q}(i),\lambda}(3-4s-4w,-\ell) \cdot \psi^{(4)}_{1}(3/2-s-w, (-1)^a \lambda^{2b} \alpha^2/m^2,-\ell),
\end{equation}
and
\begin{equation} \label{series2}
\zeta_{\mathbb{Q}(i),\lambda}(3-4s-4w,-\ell) \cdot \psi^{(4)}_{1+\lambda^3}(3/2-s-w, (-1)^{a} \lambda^{2b} \alpha^2/m^2,-\ell),
\end{equation}
for $a,b \in \mathbb{Z}_{\geq 0}$ ($a,b$ are absolutely bounded 
since $\alpha^2/m^2 \equiv 1 \ppmod{\lambda^3}$, see \eqref{indexprop}--\eqref{gammasupport} 
and \eqref{Scoprimecusp}).
After using absolute convergence of $Z_{i1}(3/2-s-w, (-1)^{a} \lambda^{2b} \alpha^2/m^2,-\ell)$
on the line $\Re(w)=-\sigma-2 \varepsilon$ 
to interchange the order of summation and integration, we see that 
\begin{equation*}
\frac{1}{2 \pi i} \int_{-\sigma-2 \varepsilon-i \infty}^{-\sigma -2 \varepsilon+i \infty} h_{\beta}(s+w,\ell;\alpha)
\Gamma(w) x^w dw
\end{equation*}
is equal to a finite sum of terms indexed by $j=1,\ldots,24$, $\eta \in \{1,1+\lambda^3\} \ppmod{4}$, and $(a,b) \in S(\kappa_j,1)$,
\begin{align*}
& N(\alpha)^{1-2s} (-1)^{C(\alpha,\alpha \beta)} \Big( \frac{\overline{\alpha}}{|\alpha|} \Big)^{\ell} \\
 & \times \sum_{m \mid \alpha} \mathscr{H}(3/2-s,\eta,\ell,\kappa_j,(a,b);\alpha,m) \frac{\mu(m)}{N(m)^s}(-1)^{C(m,m \alpha \beta)+C(m,m \alpha)} \widetilde{g}_4 \Big( \frac{\alpha}{m} \Big)
\Big( \frac{-1}{m} \Big)_4 \Big( \frac{\overline{m}}{|m|} \Big)^{3 \ell}   \\
& \times \Big( \frac{\overline{\alpha/m}}{|\alpha/m|} \Big)^{2 \ell}  
\sum_{\substack{d \in \mathbb{Z}[i] \\ d \equiv 1 \ppmod{\lambda^3}} } N(d)^{4s-3} \Big( \frac{\overline{d}}{|d|}  \Big)^{-4 \ell}
\sum_{\substack{c \in \mathbb{Z}[i] \\ c \equiv \eta \ppmod{4} }} N(c)^{s-1} \widetilde{g}_4 \Big( \frac{(-1)^{a} \lambda^{2b} \alpha^2}{m^2},c \Big) \Big( \frac{\overline{c}}{|c|} \Big)^{-\ell} \\
& \times \frac{1}{2 \pi i} \int_{-\sigma-2 \varepsilon-i \infty }^{-\sigma-2 \varepsilon+i \infty} \Big( \frac{2^b N(cd^4) x}{N(\alpha^2m)} \Big)^{w}
\frac{G_{\infty}(3/2-s-w,-\ell)}{G_{\infty}(s+w+1/2,\ell)} \Gamma(w) \\
& \times \begin{cases}
(-1)^{\ell} A_{j i_1}(2^{-s-w-1/2},\ell) & \quad \text{if} \quad m \alpha \beta \equiv 1 \ppmod{4}  \\
(-1)^{\ell+1} A_{j i_2}(2^{-s-w-1/2},\ell) & \quad \text{if} \quad m \alpha \beta \equiv 1+\lambda^3 \ppmod{4} 
\end{cases}
dw,
\end{align*} 
where the coefficients $\mathscr{H}(s,\eta,\ell,\kappa_j,(a,b);\alpha,m)$ are as stated in the Lemma.
Recalling \eqref{mellin}--\eqref{34res} we obtain the result.
\end{proof}

\begin{lemma} \label{decaylem}
Let $0<\varepsilon<1/100$ be given.
For $\ell\in\mathbb Z$, $\beta\in\{1,1+\lambda^3\}\ppmod4$, $y>0$, and
$s=\sigma+it$ with $\sigma=1/2+\varepsilon$, 
let $\mathscr{T}_{\beta}(s,y,\ell;j)$
be as in \eqref{Tbetadef}.
Then for any $2\varepsilon\leq B<1$ we have,
\begin{equation} \label{Tbound}
 \mathscr{T}_{\beta}(s,y,\ell;j) \ll_{\varepsilon,B}  (|s|+|\ell|+1)^{3-6 \sigma} \cdot (y(|s|+|\ell|+1)^6)^{B},
\end{equation}
for all $y>0$.
\end{lemma}
\begin{proof}
Let $2\varepsilon\leq B<1$.
We move the contour in \eqref{Tbetadef} to $\Re(w)=-B$.
Both contours lie in $-1<\Re(w)<0$, so no pole of $\Gamma(w)$ is crossed.
Throughout the shift, $\Re(s+w+1/2)<1$, so the factors $A_{ji}$ are holomorphic.
The numerator gamma factors also have positive real parts.
We apply Stirling's formula \cite[(5.11.1)]{NIST:DLMF} to 
estimate the quotient, 
\begin{equation}
\frac{G_{\infty}(3/2-s-w,-\ell)}{G_{\infty}(1/2+s+w,\ell)} \ll_B (|s+w|+|\ell|+1)^{3-6 \sigma+6B},
\end{equation} 
as $\Im(s+w) \rightarrow \pm \infty$.
Thus,
\begin{align} \label{Tinit}
 \mathscr{T}_{\beta}(s,y,\ell;j) & \ll_B
 \int_{-B-i \infty}^{-B+i \infty}  (|s+w|+|\ell|+1)^{3-6 \sigma} \cdot (y(|s+w|+|\ell|+1)^6)^{B} \nonumber \\  
& \times  |\Gamma(w)| \cdot \max_{i \in \{i_1,i_2\}}
 \{  |A_{ji}(2^{-s-w-1/2},\ell)| \}   |dw|.
 \end{align}
Recall that $A_{ji}(2^{-z},\ell)$ are rational functions in $2^{-z}$ that are
holomorphic when $\Re(z) \neq 1, 5/4$, so
 \begin{equation} \label{rational}
  \max_{i \in \{i_1,i_2\}}
 \{  |A_{ji}(2^{-s-w-1/2},\ell)| \} \ll_{\varepsilon,B}1 \quad \text{for} \quad 2\varepsilon\leq B=-\Re(w)<1.
 \end{equation}
Thus \eqref{Tinit} and \eqref{rational} give the result.
\end{proof}

\begin{prop} \label{secondmoment}
For $\ell \in \mathbb{Z}$, $\beta \in \{1,1+\lambda^3 \} \ppmod{4}$, 
and $\alpha \in \mathbb{Z}[i]$ squarefree with $\alpha \equiv 1 \ppmod{\lambda^3}$,
let $h_{\beta}(s,\ell;\alpha)$ be as in  
\eqref{littlegdef}.
For $u \geq 1$, $t \in \mathbb{R}$, and $0<\varepsilon<1/100$ we have 
\begin{equation*}
\sum_{\substack{N(\alpha) \leq u \\ \alpha \equiv 1 \ppmod{\lambda^3} }} \mu^2(\alpha) |h_{\beta}(1/2+\varepsilon+it,\ell;\alpha)|^2 \ll_{\varepsilon}
u^{3/2+\varepsilon} (|t|+|\ell|+1)^{3+\varepsilon}.
\end{equation*}
\end{prop}

\begin{remark}
This result is convex in the $t$ and $\ell$ aspects, but saves $u^{1/2}$ in the $\alpha$-aspect over convexity.
\end{remark}

\begin{proof}
Put $T:=|t|+|\ell|+1$, $x:=u^{3/2}T^3$, and $\delta:=\varepsilon/10$.
We use the triangle inequality in the statement of Lemma \ref{hvoronoi} with $s=1/2+\varepsilon+it$
and $\sigma=1/2+\varepsilon$, 
square, and then
sum over squarefree $\alpha \in \mathbb{Z}[i]$ with $N(\alpha) \leq u$
and $\alpha \equiv 1 \ppmod{\lambda^3}$. We obtain 
\begin{equation} \label{hsecondmoment}
\sum_{\substack{N(\alpha) \leq u \\ \alpha \equiv 1 \ppmod{\lambda^3} }} \mu^2(\alpha) |h_{\beta}(1/2+\varepsilon+it,\ell;\alpha)|^2 \ll \delta_{\ell=0} \cdot e^{-2 |t|} (ux)^{1/2+\varepsilon}+B_1(x,\ell,u,t)+B_2(x,\ell,u,t), 
\end{equation}
where 
\begin{align} \label{B1exp}
 B_1(x,\ell,u,t) &:=\sum_{\eta \in \{1,1+\lambda^3 \} \ppmod{4}}
\sum_{\substack{ N(\alpha) \leq u \\ \alpha \equiv \eta \ppmod{4} }}
\mu^2(\alpha)
\Bigg ( \sum_{\substack{m,d \in \mathbb{Z}[i] \\ m \mid \alpha \\ d \equiv 1 \ppmod{\lambda^3}}} 
\frac{1}{N(md)^{1+4 \varepsilon}} \nonumber \\
& \times \Big | \sum_{\substack{ c \in \mathbb{Z}[i] \\ c \equiv \eta \beta \ppmod{4} } } \frac{1}{N(c)^{1/2+\varepsilon+it}}
\widetilde{g}_4(c) \Big( \frac{\overline{c}}{|c|} \Big)^{\ell} 
e^{-N(cd^4 m^4)/x} \Big( \frac{\alpha}{c} \Big)_2 \Big |   \Bigg )^2,
\end{align}
and 
\begin{align} \label{B2exp}
& B_2(x,\ell,u,t)  \nonumber  \\
&:= \sum_{\substack{ N(\alpha) \leq u \\ \alpha \equiv 1 \ppmod{\lambda^3} }} \mu^2(\alpha)N(\alpha)^{-4\varepsilon}
\Bigg (\sum_{\substack{j=1,\ldots,24 \\ \eta \in \{1,1+\lambda^3\} \ppmod{4}} } \sum_{(a,b) \in S(\kappa_j,1)}
\sum_{\substack{m,d \in \mathbb{Z}[i] \\ m \mid \alpha \\ d \equiv 1 \ppmod{\lambda^3}}} 
\frac{1}{N(m)^{1/2+\varepsilon} N(d)^{1-4 \varepsilon} }   \nonumber \\ 
& \times \Big | \sum_{\substack{ c \in \mathbb{Z}[i] \\ c \equiv \eta \ppmod{4} } } \frac{1}{N(c)^{1/2-\varepsilon-it}} \widetilde{g}_4 \Big(\frac{(-1)^{a} \lambda^{2b} \alpha^2}{m^2},c \Big) \Big( \frac{\overline{c}}{|c|} \Big)^{-\ell} 
\mathscr{T}_{m \alpha \beta} \Big(\frac{1}{2}+\varepsilon+it,\frac{N(\alpha^2 m)}{2^b x N(cd^4)},\ell;j \Big) \Big | \Bigg)^2,
\end{align}
where we used \eqref{Hdefn} and \eqref{gammauniformbd} to obtain $\mathscr{H}(1-\varepsilon-it,\eta,\ell,\kappa_j,(a,b);\alpha,m) \ll 1$
uniformly in all parameters.

We first consider $B_1(x,\ell,u,t)$ in \eqref{B1exp}. We drop the condition that 
$\alpha \equiv \eta \pmod{4}$ by positivity.
We apply the Cauchy-Schwarz inequality to the $m,d$ summations, and then interchange the 
order of the $\alpha$ and $m$ summations. We obtain
\begin{align} \label{B1boundintermed}
& B_1(x,\ell,u,t) \nonumber \\
& \ll \sum_{\eta \in \{1,1+\lambda^3\} \ppmod{4}} \sum_{\substack{N(\alpha) \leq u \\ \alpha \equiv 1 \ppmod{\lambda^3}} } \mu^2(\alpha) \sum_{\substack{d,m \in \mathbb{Z}[i] \\ d \equiv 1 \ppmod{\lambda^3} \\ m \mid \alpha }}
\frac{1}{N(md)^{1+4 \varepsilon}} \nonumber \\
& \times \Big | \sum_{\substack{c \in \mathbb{Z}[i] \\ c \equiv \eta \beta \ppmod{4}} } \frac{1}{N(c)^{1/2+\varepsilon+it}}
\widetilde{g}_4(c) \Big( \frac{\overline{c}}{|c|} \Big)^{\ell}  e^{-N(c d^4 m^4)/x}
\Big(\frac{\alpha}{c} \Big)_2  \Big |^2 \nonumber \\
& \ll  \sum_{\eta \in \{1,1+\lambda^3\} \ppmod{4}} \sum_{\substack{d,m \in \mathbb{Z}[i] \\ d,m \equiv 1 \ppmod{\lambda^3}  }} \frac{\mu^2(m)}{N(md)^{1+4 \varepsilon}}
\sum_{\substack{N(\alpha^{\prime}) \leq u \\ \alpha^{\prime} \equiv 1 \ppmod{\lambda^3} } } \mu^2(\alpha^{\prime}) \nonumber \\
& \times \Big | \sum_{\substack{c \in \mathbb{Z}[i] \\ N(c) \leq x^{1+\delta/2} \\ c \equiv \eta \beta \ppmod{4}} } \mu^2(c) a_c(d,m,t,\ell) 
\Big(\frac{\alpha^{\prime}}{c} \Big)_2  \Big |^2+x^{-1000},
\end{align}
where
\begin{equation} \label{acdef}
a_c(d,m,t,\ell):=\frac{1}{N(c)^{1/2+\varepsilon+it}}
\widetilde{g}_4(c) \Big( \frac{\overline{c}}{|c|} \Big)^{\ell} \Big(\frac{m}{c}  \Big)_2  
 e^{-N(c d^4 m^4)/x}  \ll \frac{1}{N(c)^{1/2+\varepsilon}},
\end{equation}
uniformly in the parameters $d,m,t,\ell$. To obtain \eqref{B1boundintermed} we used the fact that the weights
$a_c(d,m,t,\ell)$ are supported on squarefree $c \in \mathbb{Z}[i]$ because $\widetilde{g}_4(c)$
is, and used positivity to lengthen the sum over $\alpha^{\prime}$ and drop the condition that $(\alpha^{\prime},m)=1$.
For each $d,m \in \mathbb{Z}[i]$ with $d,m \equiv 1 \ppmod{\lambda^3}$,
we apply the quadratic large sieve (Theorem \ref{quadsieve}) with a sufficiently
small auxiliary exponent in terms of $\delta$ to 
the sum over $\alpha^{\prime}$ and $c$ in \eqref{B1boundintermed}, and then use \eqref{acdef} to
obtain,
\begin{align} \label{B1bound}
 B_1(x,\ell,u,t) & \ll  (x u)^{\delta} (u+x)
\sum_{\substack{d,m \in \mathbb{Z}[i] \\ d,m \equiv 1 \ppmod{\lambda^3} }} \frac{1}{N(md)^{1+4 \varepsilon}}
\sum_{N(c) \leq x^{1+\delta/2} } \mu^2(c) |a_c(d,m,t,\ell)|^2 \nonumber \\
&+x^{-1000} \nonumber \\
& \ll (x u)^{\delta} (u+x).
\end{align}

We now consider \eqref{B2exp}. 
Since $\alpha$ is squarefree, each non-trivial prime power in $\mathbb{Z}[i]$ dividing $\alpha^2/m^2$
has exponent exactly $2$. 
Thus \eqref{quartquad}--\eqref{rel2}, \eqref{rel4}, and \eqref{sqrootcancel}, tell us that
\begin{align} \label{localeval}
& \widetilde{g}_4 \Big(\frac{(-1)^a \lambda^{2b} \alpha^2}{m^2},c \Big) \nonumber \\
&=\overline{\Big( \frac{(-1)^a \lambda^{2b}}{c} \Big)}_4
\cdot \begin{cases}
\big( \frac{\alpha/m}{c} \big)_2 \widetilde{g}_4(c) 
& \text{if } (c,\alpha/m)=1 \\
(-1)^{C(c^{\prime},e)}\big(\frac{\alpha/(me)}{ec^{\prime}}\big)_2
\widetilde{g}_4(c^{\prime})\widetilde{g}_4(e^2,e^3)
&\begin{gathered}\text{if }c=e^3c^{\prime},\quad 1\neq e\mid\tfrac{\alpha}{m},\\
(c^{\prime},\alpha/m)=1\end{gathered} \\
0 & \text{otherwise}
\end{cases},
\end{align}
for squarefree $\alpha$,
and 
\begin{equation} \label{g4localnormeval}
\frac{\big | \widetilde{g}_4 ( e^2, e^3) \big |}{N(e)^{3/2}}=N (e)^{-1/2}
\end{equation}
for $e \mid \alpha/m$.

We substitute \eqref{localeval} and \eqref{g4localnormeval} into \eqref{B2exp}, and use the triangle inequality and \eqref{sqrootcancel}--\eqref{tildedef}
to estimate trivially,
\begin{align} \label{B2exp2}
 & B_2(x,\ell,u,t)  \nonumber  \\
& \leq  \sum_{\substack{ N(\alpha) \leq u \\ \alpha \equiv 1 \ppmod{\lambda^3} }} \mu^2(\alpha)N(\alpha)^{-4\varepsilon}
\Bigg (\sum_{\substack{j=1,\ldots,24 \\ \eta \in \{1,1+\lambda^3\} \ppmod{4}} } \sum_{(a,b) \in S(\kappa_j,1)}
\sum_{\substack{m,d \in \mathbb{Z}[i] \\ m \mid \alpha \\ d \equiv 1 \ppmod{\lambda^3}}} 
\frac{1}{N(m)^{1/2+\varepsilon} N(d)^{1-4 \varepsilon} }   \nonumber \\ 
& \times \Bigg( \sum_{\substack{ c \in \mathbb{Z}[i] \\ (c,\alpha/m)=1 \\ c \equiv \eta \ppmod{4} } } \frac{1}{N(c)^{1/2-\varepsilon}}
 \Big |\mathscr{T}_{m \alpha \beta} \Big(\frac{1}{2}+\varepsilon+it,\frac{N(\alpha^2 m)}{2^b x N(cd^4)},\ell;j \Big) \Big |  \nonumber \\
 & + \sum_{\substack{e\mid(\alpha/m)\\e\neq1}}\frac{1}{N(e)^{1/2-3\varepsilon}}\cdot  \sum_{\substack{ c^{\prime} \in \mathbb{Z}[i] \\ (c^{\prime},\alpha/m)=1 \\ c^{\prime} \equiv e\eta \ppmod{4} } } 
\frac{1}{N(c^{\prime})^{1/2-\varepsilon}}
 \Big | \mathscr{T}_{m \alpha \beta} \Big(\frac{1}{2}+\varepsilon+it,\frac{N(\alpha^2 m)}{2^b x N(e^3c^{\prime}d^4)},\ell;j \Big) \Big | \Bigg)  \Bigg)^2.
\end{align}

By Lemma \ref{decaylem} with $B=2\varepsilon$ and $B=3/4$,
\begin{equation} \label{Tepsbound}
\mathscr T_{\beta}(1/2+\varepsilon+it,y,\ell;j)
\ll_{\varepsilon}T^{-6\varepsilon}
\min\{(yT^6)^{2\varepsilon},(yT^6)^{3/4}\}.
\end{equation}
For $Y>0$, we have
\begin{equation} \label{dualtailsum}
\sum_{c\equiv1\ppmod{\lambda^3}}
\frac{\min\{(Y/N(c))^{2\varepsilon},(Y/N(c))^{3/4}\}}
{N(c)^{1/2-\varepsilon}}
\ll_{\varepsilon}Y^{1/2+\varepsilon},
\end{equation}
by splitting at $N(c)=Y$. For $Y<1$ the sum is $O_{\varepsilon}(Y^{3/4})$.
Apply \eqref{Tepsbound}--\eqref{dualtailsum} with
$Y=T^6N(\alpha)^2N(m)/(xN(d)^4)$ and $Y/N(e)^3$ to the two inner sums
in \eqref{B2exp2}. The factors $2^b$ are bounded. Including
$N(\alpha)^{-2\varepsilon}$, summing $N(d)^{-3}$, and using a divisor bound,
the expression being squared is
\begin{align} \label{dualinnerbound}
&\ll_{\varepsilon} T^3N(\alpha)x^{-1/2-\varepsilon}
\sum_{m\mid\alpha}\left(1+\sum_{e\mid(\alpha/m)}N(e)^{-2}\right)\nonumber\\
&\ll_{\varepsilon}T^3N(\alpha)^{1+\delta/2}x^{-1/2-\varepsilon}.
\end{align}
Squaring and summing over $\alpha$ gives
\begin{equation} \label{B2bound}
B_2(x,\ell,u,t)\ll_{\varepsilon}
T^6x^{-1-2\varepsilon}u^{3+\delta}.
\end{equation}
With our choice $x=u^{3/2}T^3$, \eqref{B1bound} and \eqref{B2bound}
are both $O_{\varepsilon}(u^{3/2+\varepsilon}T^{3+\varepsilon})$.
The pole term in \eqref{hsecondmoment} satisfies the same bound:
when $\ell=0$, its factor $e^{-2|t|}$ absorbs the remaining powers of $T$,
and $u^{5/4+(5/2)\varepsilon}\leq u^{3/2+\varepsilon}$.
Substituting these estimates into \eqref{hsecondmoment} proves the result.
\end{proof}

\begin{remark} \label{lindelofremark}
We treated the dual sum $B_2(x,\ell,u,t)$ in \eqref{B2exp} trivially in the above proof. 
Suppose that $\alpha$ was restricted to run over primes instead of squarefree elements of $\mathbb{Z}[i]$, say.
With a more careful treatment of \eqref{B2exp} using the quadratic large sieve when $m=1$,
and the fact that $\psi^{(4)}_{\beta}(z,1,\ell)$ only has a possible simple pole at $z=5/4$ in the half-plane $\sigma>1$
($\ell=0$) when $m=\alpha$,
it seems possible to improve Proposition \ref{secondmoment} to
a Lindel\"{o}f-on-average result in the $\alpha$-aspect. Such an improvement is a moot point, 
see the discussion at the end of \S \ref{overall}.
\end{remark}

\begin{corollary} \label{type1cor}
For $\ell \in \mathbb{Z}$, $\beta \in \{1,1+\lambda^3 \}$, $X \geq 1$ and $1 \leq u <X^{1/2}$ we have 
\begin{equation*}
\Sigma_{1,\beta}(X,\ell,u), \hspace{0.1cm} \Sigma_{2^{\prime},\beta}(X,\ell,u) \ll_{\varepsilon}  X^{3/4+\varepsilon} + (X(|\ell|+1))^{\varepsilon} \cdot X^{1/2} u^{3/4} (|\ell|+1)^{3/2},
\end{equation*}
\end{corollary}
\begin{proof}
The result follows from combining Lemma \ref{type1init}, Lemma \ref{type1init2}, and Proposition \ref{secondmoment},
using both weights $R$ and $R_{\log}$ and absorbing logarithmic factors into $X^{\varepsilon}$.
\end{proof}

\section{Proof of Theorem \ref{mainthm}}
\begin{proof}[Proof of Theorem \ref{mainthm}]
Recall Vaughan's identity \eqref{Vauid} with choice $u=X^{1/3}$. 
From Lemma \ref{type1init} we have
\begin{equation*}
\Sigma_{0,\beta}(X,\ell,X^{1/3})=H_{\beta}(X,\ell), \quad \text{and} \quad\Sigma_{4,\beta}(X,\ell,X^{1/3})=0,
\end{equation*}
where $H_{\beta}(X,\ell)$ is as in \eqref{Helldefn}.
Corollary \ref{type1cor} gives the bound for the Type-I sums,
\begin{equation*}
\Sigma_{1,\beta}(X,\ell,X^{1/3}), \hspace{0.1cm} \Sigma_{2^{\prime},\beta}(X,\ell,X^{1/3})
\ll  (X(|\ell|+1))^{\varepsilon} \cdot X^{3/4} (|\ell|+1)^{3/2}.
\end{equation*}
Proposition \ref{type2prop} gives the bound for the Type-II sums 
\begin{equation}
\Sigma_{2^{\prime \prime},\beta}(X,\ell,X^{1/3}), \hspace{0.1cm} \Sigma_{3,\beta}(X,\ell,X^{1/3}) \ll X^{5/6+\varepsilon}, 
\end{equation}
uniformly in $\ell$. The result now follows.
\end{proof}

\bibliographystyle{amsalpha}
\bibliography{quartic} 

\providecommand{\bysame}{\leavevmode\hbox to3em{\hrulefill}\thinspace}
\providecommand{\MR}{\relax\ifhmode\unskip\space\fi MR }
\providecommand{\MRhref}[2]{%
  \href{http://www.ams.org/mathscinet-getitem?mr=#1}{#2}
}
\providecommand{\href}[2]{#2}
\begin{thebibliography}{Kum75}

\bibitem[BH89]{BH}
D.~Bump and J.~Hoffstein, \emph{Some conjectured relationships between theta
  functions and {E}isenstein series on the metaplectic group}, Number theory
  ({N}ew {Y}ork, 1985/1988), Lecture Notes in Math., vol. 1383, Springer,
  Berlin, 1989, pp.~1--11. \MR{1023915}

\bibitem[BH16]{BrHo}
R.~Br\"{o}ker and J.~Hoffstein, \emph{Fourier coefficients of sextic theta
  series}, Math. Comp. \textbf{85} (2016), no.~300, 1901--1927. \MR{3471113}

\bibitem[CFH12]{CFH}
G.~Chinta, S.~Friedberg, and J.~Hoffstein, \emph{Double {D}irichlet series and
  theta functions}, Contributions in analytic and algebraic number theory,
  Springer Proc. Math., vol.~9, Springer, New York, 2012, pp.~149--170.
  \MR{3060459}

\bibitem[Del80]{Deli}
P.~Deligne, \emph{Sommes de {G}auss cubiques et rev\^{e}tements de {${\rm
  SL}(2)$} [d'apr\`{e}s {S}. {J}. {P}atterson].}, S\'{e}minaire {B}ourbaki
  (1978/79),, Lecture Notes in Math., 770,, ,, 1980, pp.~Exp. No. 539, pp.
  244--277,. \MR{572428}

\bibitem[Dia04]{Dia}
A.~Diaconu, \emph{Mean square values of {H}ecke {$L$}-series formed with
  {$r$}-th order characters}, Invent. Math. \textbf{157} (2004), no.~3,
  635--684. \MR{2092772}

\bibitem[{\relax DLMF}]{NIST:DLMF}
\emph{{\it NIST Digital Library of Mathematical Functions}},
  http://dlmf.nist.gov/, Release 1.0.18 of 2018-03-27, F.~W.~J. Olver, A.~B.
  {Olde Daalhuis}, D.~W. Lozier, B.~I. Schneider, R.~F. Boisvert, C.~W. Clark,
  B.~R. Miller and B.~V. Saunders, eds.

\bibitem[DR24]{DR}
A.~Dunn and M.~Radziwi{\l\l}, \emph{Bias in cubic {G}auss sums: {P}atterson's
  conjecture}, Ann. of Math. (2) \textbf{200} (2024), no.~3, 967--1057.
  \MR{4816436}

\bibitem[EP92]{EckPat}
C.~Eckhardt and S.~J. Patterson, \emph{On the {F}ourier coefficients of
  biquadratic theta series}, Proc. London Math. Soc. (3) \textbf{64} (1992),
  no.~2, 225--264. \MR{1143226}

\bibitem[Gau73]{Gauss}
C.~F. Gauss, \emph{Werke. {B}and {II}}, Georg Olms Verlag, Hildesheim, 1973,
  Reprint of the 1863 original. \MR{616130}

\bibitem[GL13]{GL}
L.~Goldmakher and B.~Louvel, \emph{A quadratic large sieve inequality over
  number fields}, Math. Proc. Cambridge Philos. Soc. \textbf{154} (2013),
  no.~2, 193--212. \MR{3021809}

\bibitem[HB95]{HB1}
D.~R. Heath-Brown, \emph{A mean value estimate for real character sums}, Acta
  Arith. \textbf{72} (1995), no.~3, 235--275. \MR{1347489}

\bibitem[HB00]{HB}
\bysame, \emph{Kummer's conjecture for cubic {G}auss sums}, Israel J. Math.
  \textbf{120} (2000), no.~part A, 97--124. \MR{1815372}

\bibitem[HBP79]{HBP}
D.~R. Heath-Brown and S.~J. Patterson, \emph{The distribution of {K}ummer sums
  at prime arguments}, J. Reine Angew. Math. \textbf{310} (1979), 111--130.
  \MR{546667}

\bibitem[Hof93]{Hoff}
J.~Hoffstein, \emph{Eisenstein series and theta functions on the metaplectic
  group}, Theta functions: from the classical to the modern, CRM Proc. Lecture
  Notes, vol.~1, Amer. Math. Soc., Providence, RI, 1993, pp.~65--104.
  \MR{1224051}

\bibitem[IR90]{IR}
K.~Ireland and M.~Rosen, \emph{A classical introduction to modern number
  theory}, second ed., Graduate Texts in Mathematics, vol.~84, Springer-Verlag,
  New York, 1990. \MR{1070716}

\bibitem[Jut75]{Jut}
M.~Jutila, \emph{On mean values of {D}irichlet polynomials with real
  characters}, Acta Arith. \textbf{27} (1975), 191--198. \MR{406954}

\bibitem[KP84]{KP}
D.~A. Kazhdan and S.~J. Patterson, \emph{Metaplectic forms}, Inst. Hautes
  \'{E}tudes Sci. Publ. Math. (1984), no.~59, 35--142. \MR{743816}

\bibitem[Kub69]{Kub1}
T.~Kubota, \emph{On automorphic functions and the reciprocity law in a number
  field}, Lectures in Mathematics, Department of Mathematics, Kyoto University,
  No. 2, Kinokuniya Book Store Co., Ltd., Tokyo, 1969. \MR{0255490}

\bibitem[Kub71]{Kub2}
\bysame, \emph{Some number-theoretical results on real analytic automorphic
  forms}, Several complex variables, {II} ({P}roc. {I}nternat. {C}onf., {U}niv.
  {M}aryland, {C}ollege {P}ark, {M}d., 1970), Lecture Notes in Math., Vol. 185,
  Springer, Berlin, 1971, pp.~87--96. \MR{0314768}

\bibitem[Kum75]{Kummer}
E.~E. Kummer, \emph{Collected papers}, Springer-Verlag, Berlin-New York, 1975,
  Volume I: Contributions to number theory, Edited and with an introduction by
  Andr\'{e} Weil. \MR{0465760}

\bibitem[Ono09]{On}
K.~Onodera, \emph{Bound for the sum involving the {J}acobi symbol in {$\Bbb
  Z[i]$}}, Funct. Approx. Comment. Math. \textbf{41} (2009), no.~part 1,
  71--103. \MR{2568797}

\bibitem[Pat77]{Pat1}
S.~J. Patterson, \emph{A cubic analogue of the theta series}, J. Reine Angew.
  Math. \textbf{296} (1977), 125--161. \MR{563068}

\bibitem[Pat78]{Pat4}
\bysame, \emph{On the distribution of {K}ummer sums}, J. Reine Angew. Math.
  \textbf{303(304)} (1978), 126--143. \MR{514676}

\bibitem[Pat84]{Pat2}
\bysame, \emph{Whittaker models of generalized theta series}, Seminar on number
  theory, {P}aris 1982--83 ({P}aris, 1982/1983), Progr. Math., vol.~51,
  Birkh\"{a}user Boston, Boston, MA, 1984, pp.~199--232. \MR{791596}

\bibitem[Pat87]{Pat3}
\bysame, \emph{The distribution of general {G}auss sums and similar arithmetic
  functions at prime arguments}, Proc. London Math. Soc. (3) \textbf{54}
  (1987), no.~2, 193--215. \MR{872805}

\bibitem[Sel63]{Sel}
A.~Selberg, \emph{Discontinuous groups and harmonic analysis}, Proc.
  {I}nternat. {C}ongr. {M}athematicians ({S}tockholm, 1962), Inst.
  Mittag-Leffler, Djursholm, 1963, pp.~177--189. \MR{0176097}

\bibitem[Suz83]{Suz1}
T.~Suzuki, \emph{Some results on the coefficients of the biquadratic theta
  series}, J. Reine Angew. Math. \textbf{340} (1983), 70--117. \MR{691962}

\bibitem[Vau75]{Vau}
R.~C. Vaughan, \emph{Mean value theorems in prime number theory}, J. London
  Math Soc. (2) \textbf{10} (1975), 153--162. \MR{0376567}

\end{thebibliography}
\end{document}